\documentclass[12pt,reqno]{amsart}

\usepackage{color}
\usepackage{amsmath, amsthm, amssymb}
\usepackage{amsfonts}
\usepackage[ansinew]{inputenc}
\usepackage[dvips]{epsfig}
\usepackage{graphicx}
\usepackage[english]{babel}
\usepackage{hyperref}

\usepackage{cite}
\usepackage{graphicx}
\usepackage{amscd}
\usepackage{bm}
\usepackage{enumerate}

\usepackage{verbatim}
\usepackage{hyperref}
\usepackage{amstext}
\usepackage{latexsym}
\usepackage{mathrsfs}

\theoremstyle{plain}
\newtheorem{Theorem}{Theorem}[section]
\newtheorem{Definition}[Theorem]{Definition}

\newtheorem{Proposition}[Theorem]{Proposition}
\newtheorem{Lemma}[Theorem]{Lemma}
\newtheorem{Remark}[Theorem]{Remark}

\numberwithin{Theorem}{section}
\numberwithin{equation}{section}

\def\proof{\noindent{{\bf Proof. }}}
\def\square{\vbox{
\hrule height .4pt \hbox{\vrule width .4pt height 7pt \kern 7pt
\vrule width .4pt} \hrule height .4pt }}
\def\QED{\hfill {$\square$}\goodbreak \medskip}

\newcommand{\average}{{\mathchoice {\kern1ex\vcenter{\hrule height.4pt
width 6pt depth0pt} \kern-9.7pt} {\kern1ex\vcenter{\hrule
height.4pt width 4.3pt depth0pt} \kern-7pt} {} {} }}

\def\R{\mathbb{R}}

\renewcommand{\a }{\alpha }

\renewcommand{\phi}{\varphi}

\newcommand{\D }{\Delta }

\newcommand{\e }{\varepsilon }

\newcommand{\n }{\nabla }

\newcommand{\s }{\sigma }

\renewcommand{\t }{\tau }

\renewcommand{\O }{\Omega }

\newcommand{\ov}{\overline}

\newcommand{\be}{\begin{equation}}
\newcommand{\ee}{\end{equation}}

\newcommand{\N}{\mathbb{N}}

\newcommand{\cG}{{\mathcal G}}

\newcommand{\cL}{{\mathcal L}}

\newcommand{\cO}{{\mathcal O}}

\DeclareMathOperator{\Iso}{{\mathcal I}}

\newcommand{\B}{{\bf B}}

\newcommand{\eps}{\varepsilon}

\renewcommand{\epsilon}{\varepsilon}

\begin{document}

\title[Exceptional domains in higher dimensions ]
{Exceptional domains in higher dimensions }

\author{Ignace Aristide Minlend}
\address{I.A.M: Faculty of Economics and Applied Management, University of Douala,  BP 2701, Douala, Littoral Province, Cameroon}
\email{\small{ignace.a.minlend@aims-senegal.org,ignace.minlend@univ-douala.com} }

\author{Tobias Weth}
\address{T.W.:  Goethe-Universit\"{a}t Frankfurt, Institut f\"{u}r Mathematik.
Robert-Mayer-Str. 10 D-60629 Frankfurt, Germany.}
\email{weth@math.uni-frankfurt.de}
\author{Jing Wu}
\address{J.W.: Departamento de An\'alisis matem\'atico, Universidad de Granada,
	Campus Fuente-nueva, 18071 Granada, Spain.}
\email{jingwulx@correo.ugr.es}


\begin{abstract}
  We prove  the existence of nontrivial unbounded  exceptional domains in the Euclidean space $\R^N$,  $N\geq4$. These domains arise as perturbations of complements of straight cylinders in $\R^N$, and by definition they support a positive harmonic function with vanishing Dirichlet boundary values and constant Neumann boundary values, the so-called roof function. While the domains have a similar shape as those constructed in the recent work  \cite{Fall-MinlendI-Weth3} for the case $N=3$, there is a striking constrast with regard to the shape of corresponding roof functions which are bounded for $N \ge 4$. Moreover, while the analysis in  \cite{Fall-MinlendI-Weth3} does not extend to higher dimensions, the approach of the present paper depends heavily on the assumption $N \ge 4$.
\end{abstract}
\maketitle

\section{Introduction and main result}
This paper deals with the  existence  of  \emph{nontrivial exceptional}  subdomains  of  the Euclidean space  $\R^{N}$,  $N\geq 4$. In the classical literature, a smooth domain $\O$ of the Euclidean space $\R^N$ is  said  to be an  \emph{exceptional} domain if  there exists a positive harmonic function   with  vanishing  Dirichlet boundary data  and  constant nonzero Neumann boundary  data. Such a function is referred to as the \emph{roof function.} The  problem of  finding  exceptional domains  goes back to the pioneer work  by  L. Hauswirth, F. H\'{e}lein, and F. Pacard  in  \cite{hauswirth-et-al}, where the  nontrivial exceptional domain
\begin{equation}
  \label{eq:catenoid-ex}
\Omega_0:= \{(x,y) \in \R^2: |y  |<\frac{\pi}{2}+\cosh (x)\}
\end{equation}
was  discovered  in the plane. Later on,  the  classification of planar exceptional domains  was addressed  by  Khavinson, Lundberg and Teodorescu \cite{KLT13}. Inspired by  \cite{KLT13},  Traizet \cite{Traizet}    was able  to  prove   that the only examples up to rotation and translation  of planar exceptional domains having finitely boundary components are the exterior of a disk, a halfplane and the nontrivial domain $\Omega_0$. In  \cite[Example 7.3]{Traizet}, he further proved the existence of a nontrial periodic  exceptional domain corresponding  to Scherk's simply periodic minimal bigraphs. We note that  a  family of infinitely connected  planar exceptional domains was already  discovered   in fluid dynamics by G. R. Baker, P. G. Saffman and J. S. Sheffield  \cite{Baker, CrowdyGreen}.

Up to date the  structure of the set  of exceptional domains in dimensions $N \ge 3$ remains largely unknown.  To mention the few existing results, the authors in  \cite[Theorem 7.1]{KLT13}  classified the exteriors of balls as the only exceptional domains  in  $\R^N$ whose complements  are  bounded, connected and have $C^{2, \alpha}$ boundaries.
Regarding the existence of \emph{nontrivial} exceptional  domains  in  higher  dimensions, we are only aware of the recent works   \cite{Fall-MinlendI-Weth3, LiuWangWei}.  Indeed,   Liu,  Wang and Wei \cite{LiuWangWei} constructed
higher dimensional analogues of the domain $\Omega_0$ given in (\ref{eq:catenoid-ex}). Moreover, in \cite{Fall-MinlendI-Weth3},   Fall, Minlend and Weth constructed exceptional subdomains $\Omega \subset \R^3$ which are perturbations of the straight cylinder in $\R^3$. These domains solve an electrostatic problem, as they enjoy the property that the constant charge distribution on their boundaries  is an electrostatic equilibrium. 
In this work, we investigate  the existence of exceptional  domains of  $\R^{N}$  with $N\geq 4$, where the overdetermined boundary value problem
\begin{align}\label{eq:ovder}
\begin{cases}
\Delta   u    = 0& \quad  \textrm { in } \quad  \Omega,  \\
   u  =  1& \quad\textrm{ on } \quad\partial \Omega, \\
 \lim \limits_{|z| \rightarrow \infty}  u (z, t )  =  0& \quad \textrm{uniformly in $t \in \R$,}\\
  \dfrac{\partial  u } {\partial \eta}= c &\quad
\textrm { on } \quad \partial\Omega,
\end{cases}
 \end{align}
is solvable and $1-u$ is the roof function associated with $\Omega$.  Here, $\eta$  is the unit outward normal vector to  the boundary $\partial \O$,  $c$ is a positive constant.
For an exterior domain $\O = \R^N \setminus \overline D$ associated with a smooth bounded domain $D$, prototypes  of problem  \eqref{eq:ovder} arise in potential theory and  are related to the \emph{Newton capacity} of $D$ defined,  for $N\geq3$,   by
$$
\operatorname{Cap}(D) = \frac{1}{N(N-2)\omega_N} \inf \left \{
\int_{\R^N} |\nabla u|^2 dz : u \in H^1_0(\R^N),  u \geq 1 \quad \textrm{in}\quad D
\right \}, $$
where $\omega_N$ is the Euclidean volume of an $N$-dimensional unit
ball (see for instance \cite{GoldmanNovagaRuffuni}). Standard results in potential theory imply this infimum is
realized by the equilibrium potential function $u_\O$, which solves
the boundary value problem
\begin {equation} \label{cap-pde}
\Delta  u_\O = 0  \textrm{ in } \Omega , \qquad
\left. u_\O \right |_{\partial \O} = 1, \qquad \lim \limits_{|z| \rightarrow \infty}
u_\O (z) = 0 \end {equation}
When $N=3$, $\operatorname{Cap}(D)$ represents the capacitance (i.e ability to hold electric charge) of the condenser $D$ immersed in an isotropic dielectric, that is the total charge $D$ can hold while maintaining a given potential energy (computed with an idealized ground at infinity). Of particular interest regarding problem  \eqref{cap-pde}   is  the question whether there exists a domain $\O$ such that the intensity of the corresponding electric field $\n u_\O$ is constant on the boundary, that is
\begin{equation}\label{Newmanncond}
|\n u_\O|=c \quad \textrm{on}\quad \partial \O,
\end{equation}
where $c$ is a positive constant.
The first  result addressing problem \eqref{cap-pde}-\eqref{Newmanncond}  was  obtained by   Reichel  \cite{R97} who proved, more generally, that if there exists a   $u\in C^2(\overline{\R^N \backslash  \O})$ with
\begin{align}\label{eq:oReichel}
\begin{cases}
-\Delta   u    = f(u)& \quad  \textrm { in } \quad  \Omega,  \\
   u  =  1& \quad\textrm{ on } \quad\partial \Omega, \\
 \lim \limits_{|z| \rightarrow \infty}  u (z )  =  0,\\
  \dfrac{\partial  u } {\partial \eta}= c&\quad
\textrm { on } \quad \partial\Omega,\\
0<u\leq a & \quad  \textrm { in } \quad  \Omega,
\end{cases}
 \end{align}
then $\O$  is ball and  $u$ is radially symmetric and radially decreasing with respect to the center of $\O$. Here $t \mapsto f(t)$ is a locally Lipschitz function, non-increasing for non negative and small values of $t$. In the special case $f\equiv 0$, this result therefore characterizes the ball as the only  electric conductor such that (when embedded in an isotropic dielectric) the intensity of the corresponding electric field is constant on the boundary. To prove his result, Reichel  \cite{R97}  used a variant of the \emph{moving plane method}. The method goes back to the work of Alexandrov
\cite{Alexandrov} on constant mean curvature surfaces, and it has been refined by Serrin in the seminal paper  \cite{S71} to prove a rigidity result for the overdetermined problem
\begin{equation}\label{eq1}
  \Delta u +f(u)=0,\: u >0 \quad \text{in $\Omega$,}\qquad\quad   u=0,\: \partial_{\eta} u=\mbox{constant}\quad \text{on $\partial\Omega$}
\end{equation}
in bounded domains $\Omega$. We refer the reader to  the references \cite{Alessandrini, BerchioGazzolaWeth, BrockHenrot,FraGazzolaKawohl, FragalaGazzola, farina-valdinoci, farina-valdinoci:2010-1,farina-valdinoci:2010-2,farina-valdinoci:2013-1, farina-valdinoci:2013-2, Gazzola, GarofaloLewis, Greco, Lamboley, MorabitoSicbaldi, Philippin, PaynePhilippin, PhilippinPayne, Rei, Fall-MinlendI-Weth, Fall-MinlendI-Weth2, Fall-Minlend} where the moving plane method has been extended to other symmetry problems and different ambient spaces.

Closely related to \cite{R97} is the work \cite{AftalionBusca} of Aftalion and  Busca, who studied  \eqref{eq:oReichel}   without  the decay at infinity  and under different assumptions on $f$ including the interesting case $f(t)=t^p$ with $\frac{N}{N-2}<p\leq \frac{N+2}{N-2}$. Moreover, Sirakov  \cite{Sirakov}   proved that the result in \cite{R97}  holds  without the assumption $u< a$ in $\O$ and for possibly multi-connected sets $\O$.

While the previous literature  is  mainly  focused on  exterior domains,  we are not aware  of any   result addressing problem  \eqref{cap-pde}-\eqref{Newmanncond}  in $\R^{N}\setminus \O,$ with $\O$ unbounded.  In this work, we are concerned with the construction of subdomains $\Omega \subset \R^{N-1}\times \R$, $N\geq 4$,  such that the overdetermined boundary value problem \eqref{eq:ovder} is solvable. The domains  we consider are   the  complements of perturbed cylinders. More precisely, they are of the form
\begin{equation}
  \label{eq:def-o-phi}
\Omega_{T, \varphi}:=\left\{ (z,t) \in \R^{N-1} \times \R\::\:  |z|>1 + \phi(\frac{2\pi}{T} t) \right\} \subset\R^N,
\end{equation}
where $T>0$ and  $\phi: \R \to (-1,1)$ is a $2 \pi$-periodic function of class $C^{2,\alpha}$, for some $\alpha \in (0.1)$. The case $\phi \equiv 0$  in   \eqref{eq:def-o-phi} corresponds to  the exterior of  the straight cylinder  $\B_{1} \times \R$, and in this case the function $u_{1}(z)=|z|^{3-N}$ solves  \eqref{eq:ovder}  with  $c=-(3-N)$.

Our main result  is the following.
\begin{Theorem}\label{Theo1}
Let $N \geq 4$. Then there exist a number  $T_*> \dfrac{2\pi}{\sqrt{N-2}}$  and  a smooth curve
$$
(-{\e},{\e}) \to   (0,+\infty) \times  C^{2,\alpha}(\R) ,\qquad s \mapsto (T_s,v_s)
$$
satisfying
$T_0= T_*$ and  $v_{0} \equiv 0$  with  $$\int^{\pi}_{-\pi} v_s (t) \cos(t)\,dt=0$$ such that for all $s\in (-{\e},{\e})$, letting  $\varphi_s(t)=s \cos(t)+s v_s$, there  exists  a unique  function $u_s\in C^{2,\a}(\ov{\O_{ T_s, \varphi_s}})$ satisfying
%
\begin{equation}\label{eq:solved-main-ND}
    \begin{cases}
       \D u_s =  0 & \quad  \textrm { in } \quad  \O_{ T_s, \varphi_s},\\
             u_s=1 & \quad  \textrm { on } \quad \partial \O_{ T_s, \varphi_s},\\
       |\n u_s | =N-3 & \quad  \textrm { on } \quad \partial \O_{ T_s, \varphi_s},\\
\lim \limits_{|z| \rightarrow \infty}  u_s (z, t )  =  0  &  \quad \textrm{uniformly in  $t \in \R$}.
    \end{cases}
\end{equation}
Moreover, $u_s$ is radial in $z$ and $T_s$-periodic and even in $t$ for every $s \in (-\eps,\eps)$.
\end{Theorem}

We point out that, for every $s\in (-{\e},{\e}) $, the domain   $\O_{ T_s, \varphi_s}$ in Theorem~\ref{Theo1} is exceptional with \emph{roof function} given by $\tilde{u}_s=1-u_s$ in $\O_{ T_s, \varphi_s}$. Indeed this function is positive in $\O_{ T_s, \varphi_s}$ since the harmonic function $u_s$ cannot attain a maximum in  $\O_{ T_s, \varphi_s}$  unless it is constant, which is excluded by  the boundary conditions in  \eqref{eq:solved-main-ND}.  Since in addition $u_s=1$ on $\partial \O_{ T_s, \varphi_s}$ and $u_s$ is $T_s$-periodic in the $t$ direction, it follows that $0<u_s<1$ and therefore $0 < \tilde{u}_s <1$ in  $\O_{ T_s, \varphi_s}$.

As already mentioned in the abstract, the domains in Theorem~\ref{Theo1} have a similar shape as those found in the recent work  \cite{Fall-MinlendI-Weth3} for the case $N=3$, but the underlying construction is completely complementary. In fact, the approach in \cite{Fall-MinlendI-Weth3} relies on specific properties of an integral representation of an associated Dirichlet-to-Neumann operator which is only available in the case $N=3$. On the other hand, our approach depends essentially on the assumption $N \ge 4$ (see e.g. Theorem~\ref{Oper-inv} and Proposition~\ref{outer2}). The difference between these two cases is reflected by the geometry of associated roof functions which are bounded for $N \ge 4$ and have a logarithmic growth in the distance from the cylinder axis in the case $N=3$. Clearly, these differences are related to the different nature of
the fundamental solution of $-\Delta$ in dimensions $N=2$ and $N \ge 3$.

Also related to  Theorem  \ref{Theo1} are some recent   results in  \cite{F.Morabito, Fall-MinlendI-Jesse}.  In \cite{F.Morabito}, Morabito obtained a family of bifurcation branches of domains which are small deformation of a solid cylinder in $\R^3$.  We note that  in contrast   to  \eqref{eq:solved-main-ND},  \cite{F.Morabito}  considers a non-constant  Neumann condition  involving  the mean curvature of the boundary.  Furthermore,  Fall, Minlend and Jesse  in \cite{Fall-MinlendI-Jesse} proved the existence of a foliation by perturbations of large coordinate sphere whose enclosures solve  \eqref{eq:solved-main-ND}  in asymptomatically  flat manifold. 

We should  also underline that  investigations toward the prototype problem  \eqref{eq1} in unbounded domains have been  mostly  motivated by  the  conjecture of Berestycki, Caffarelli and Nirenberg  \cite[p. 1110]{BeCahNi}  which states  that,  if  $\O$  is a  domain such that $\R^{N}\setminus \ov \O$ is connected, the existence  of a bounded positive solution to problem (\ref{eq1}) for some Lipschitz function $f$ implies  that  $\O$  should be a half-space, a ball, the complement of a ball, or a circular-cylinder-type domain $\R^j\times C$ (up to rotation and translation), where $C$ is a ball or a complement of a ball in $\R^{N-j}.$
This conjecture has been disproved for $N\geq3$ in \cite{S10}, where the author found a periodic perturbation of the straight cylinder $B^{N-1}\times\mathbb{R}$  supporting  a periodic solution  of (\ref{eq1}) with $f(u)=\lambda u, \lambda>0$. Further results  for   \eqref{eq1} have also  been obtained for instance in \cite{DPW15, Fall-MinlendI-Weth, RRS20, RSW21, SS12, RSW22}  and we highlight the recent contribution  by  Ruiz,  Sicbaldi and Wu \cite{SS12}, where bifurcation branches have been obtained  in onduloid type domains for a very general class of  nonlinearities $f : [0, +\infty)\to \R$.  Positive  results  addressing  the conjecture were found by Farina and Valdinoci  \cite{farina-valdinoci:2010-1} with the  assumptions  under which an  epigraph $\O$  admitting  a solution to  problem (\ref{eq1}) must be a half-space. In  \cite{RS13}, it is also shown that the conjecture remains valid in the plane for some classes of nonlinearities $f$.   Moreover,  \cite{RRS17} proves the above conjecture in dimension $2$ if $\partial\Omega$ is connected and unbounded.\\

We now explain the construction of the domains in  \eqref{eq:def-o-phi} while presenting the contents of the paper.  In Section  \ref{pullback}, we rephrase  the main problem   \eqref{eq:ovder} to the  equivalent problem (see  \eqref{eq:ovder2rephrased}-\eqref{eq:overcondi22}) on the fixed domain $\O_1=\B^c_{1} \times \R$.  We emphasize that our analysis strongly relies on the  decay assumption in  \eqref{eq:ovder}  which motivates the functional setting   in Section  \ref{sec:dirichlet}, where we work in  weighted H\"{o}lder  spaces. Furthermore,  we need   the pull back  operator in Lemma  \ref{eq:pulbal} to   map between these spaces. We do this  by  parametrizing  the set   $\O_{T,\varphi}$  as  in \eqref{eq:param}  with a suitably chosen diffeomorphism which minimizes the effect of the perturbation away from the boundary. In this functional analytic setting we are able to reformulate, in Section~\ref{sec:soln_construction}, our problem as a nonlinear operator equation of the form $F(T,\varphi)\equiv0$, to which we then apply the Crandall-Rabinowitz bifurcation theorem (see \cite{M.CR}). For this it is necessary to compute the linearised operator  $D_\varphi F (T,0)$ and analyze its spectral properties, which we do in Section \ref{linearizaton}. The proof of Theorem~\ref{Theo1} is then completed in Section~\ref{sec:Main result}. The paper ends with an appendix  where we collect some useful scale-invariant H\"{o}lder estimates for solutions of the Poisson equation and properties of modified Bessel functions.\\

\noindent \textbf{Acknowledgements}:
I.A.M. is   supported by the Alexander von Humboldt foundation and  J.W. is  supported by the China Scholarship Council (CSC201906290013) and by J. Andalucia (FQM-116).
Part of this work was  carried  out when  I.A.M.  and J.W. were visiting the Goethe University Frankfurt am Main. They are grateful to the Mathematics department  for the hospitality.

\section{The pull-back problem}\label{pullback}
We begin by fixing some notation. For $k \in \mathbb{N }\cup \{0\}$, we let
$$
C^{k,\alpha}_{p, e}( \R)
:= \left \{ u \in  C^{k,\alpha}( \R): \quad  u \textrm{ is  $2\pi$-periodic and even}\right \}.
$$
Moreover, we consider the open set
$$
\mathcal{U}:=\{\varphi \in  C^{2,\alpha}_{p,e}(\R):  \|\phi\|_\infty<1 \}.
$$
Recalling our problem, we are looking for a  number $T>0$ and a nonconstant function $\phi \in \mathcal{U}$ such that the overdetermined problem
\begin{align}\label{eq:ovder3}
\begin{cases}
\Delta   u    = 0& \quad  \textrm { in } \quad   \Omega_{T, \varphi} \\
   u  =  1& \quad\textrm{ on } \quad\partial  \Omega_{T, \varphi} \\
 \lim \limits_{|z| \rightarrow \infty}  u (z, t )  =  0& \quad \textrm{uniformly in $t \in \R$}\\
  \dfrac{\partial  u } {\partial \eta}= N-3,&\quad
\textrm { on } \quad \partial \Omega_{T, \varphi}
\end{cases}
 \end{align}
is solvable in  the perturbed domain  $\Omega_{T, \varphi}$ defined in  \eqref{eq:def-o-phi}.


In order to find a suitable variational framework for this problem, it is convenient to pull-back \eqref{eq:ovder3} on the fixed domain $\O_1:=\B^c_1\times \R$ via a suitable diffeomorphism of the form
\begin{equation}
  \label{eq:general-diff}
\O_1  \to   \Omega_{T, \varphi},\qquad (y, \t)  \mapsto  \Bigl((1  + \phi(\tau)|y|^s)y,\frac{T}{2\pi}\tau\Bigr).
\end{equation}
Note that, since the function $r \mapsto (1+c r^s)r$ is strictly increasing on $(1,\infty)$ for $|c| \le 1$ and $0\geq s \ge -2$, it is easy to see that (\ref{eq:general-diff}) defines a diffeomorphism if $T>0$, $0\geq s \ge -2$ and $\|\phi\|_\infty \le 1$. It will turn out to be important in our functional analytic framework to minimize the effect of $\phi$ for large values of $|y|$, which leads us to choose $s = -2$ in the following. Hence we consider, for $T>0$ and  a $2 \pi$-periodic positive  function $\phi \in C^{2,\alpha}(\R)$ with $\|\phi\|_\infty<1$, the diffeomorphism
\begin{equation}\label{eq:param}
\Psi_{T,\varphi}:   \O_1  \to   \Omega_{T, \varphi}, \quad  (y, \t)  \mapsto \bigl(\kappa ( |y|^2,  \varphi(\tau))y, \frac{T}{2\pi}\tau\bigr)
\end{equation}
with
\begin{align}\label{eq:defkappa}
\kappa(a,  b)=1+\frac{b}{a},   \qquad \text{for $a \geq 1$ and $|b| \le 1$ .}
\end{align}
Furthermore,
\begin{align}\label{ebiject}
 z=\kappa ( |y|^2,b)y \Longleftrightarrow y= \zeta( |z|^2,b)z,
\end{align}
where $\zeta$ is the unique function given by
\begin{align}\label{zetafunc}
\zeta( a , b)=\frac{1}{2}+\sqrt{\frac{1}{4}-\frac{b}{a}}, \qquad \text{for $|b| \le 1$ and $a \ge 4 b.$}
\end{align}

In order pull back the problem \eqref{eq:ovder3} on $\O_1,$ we consider the  ansatz
\begin{align}\label{eqans}
u(z,t)= w(\zeta( |z|^2,\varphi(\frac{2\pi}{T}t))z,\frac{2\pi}{T}t) =w(y, \tau) \qquad \text{for some function $w:  \O_1   \to \R$,}
\end{align}
and we  look for  the operator $L_{T, \varphi}$ such that
\begin{align}\label{eqrpulll}
L_{T, \varphi} w(y, \tau ) = \Delta u(z,t) \qquad \text{for $(z,t) \in  \Omega_{T, \varphi}$}
\end{align}
where $\tau=\frac{2\pi}{T} t$. To write down the operator $L_{T, \varphi}$ explicitly, we need the partial derivatives $\zeta_i = \partial_i \zeta$, $\zeta_{ii} = \partial_{ii} \zeta$, $i=1,2$ of the function $\zeta$ in (\ref{zetafunc}), which are given as follows for $|b| < 1$ and $a > 4 b$:
\begin{align}\label{derivativeszet}
\begin{cases}
\zeta_1(a,b)=\dfrac{b}{2a^2(\zeta(a,b)-\frac{1}{2})}\\
\zeta_{11}(a,b)=-\dfrac{b}{a^3(\zeta(a,b)-\frac{1}{2})}\biggl(1+\dfrac{b}{4a(\zeta(a,b)-\frac{1}{2})^2}\biggl)\\
\zeta_2(a,b)=-\dfrac{1}{2a(\zeta(a,b)-\frac{1}{2})}\\
\zeta_{22}(a,b)=-\dfrac{1}{4a^2(\zeta(a,b)-\frac{1}{2})^3},
\end{cases}
 \end{align}

\begin{Lemma}\label{eq:pulbal} 
For every  $T>0$ and  $\varphi \in \mathcal{U}$, the  operator $L_{T, \varphi} $ in \eqref{eqrpulll} is given by
\begin{align}\label{eqropepulbackkk}
L_{T, \varphi}w &= \zeta^2 \Delta_y w+\left(\frac{2\pi}{T}\right)^2\frac{\partial^2 w}{ \partial \tau^2}+  \frac{4 }{\zeta^2}  \Biggl ( \zeta_{1} \Bigl(\zeta+  |y|^2  \frac{\zeta_{1} }{\zeta^2}\Bigl) + \Bigl(\frac{2\pi}{T}\Bigl)^2   \zeta^2_{2} \varphi'^2    \Biggl)   \sum^{N-1}_{k, \ell=1}  y_\ell  y_k  \frac{\partial^2  w}{\partial y_\ell \partial y_k}\nonumber\\
&\quad+\Biggl(2(N+1)\frac{\zeta_1}{\zeta}+ 4  \frac{\zeta_{11} }{\zeta^3}|y|^2+ \Bigl(\frac{2\pi}{T}\Bigl)^2 \Bigl( \frac{ \varphi'' \zeta_{2}+ \varphi'^2 \zeta_{22} }{\zeta}   \Bigl)\Biggl)y \cdot \nabla_y w \nonumber\\
                &\quad+\left(\frac{2\pi}{T}\right)^2  \frac{\zeta_{2} \varphi' }{\zeta}  \sum^{N-1}_{\ell=1}  y_\ell \frac{\partial^2  w}{\partial y_\ell \partial \tau},
\end{align}
where, for abbreviation, we merely write $\zeta$ in place of the function
\begin{equation}
\label{eq:zeta-replaced-y-tau}
(y,\tau) \mapsto \zeta(|\kappa(|y|^2,\phi(\tau))y|^2,\phi(\tau))
\end{equation}
and similarly for the partial derivatives $\zeta_i$, $\zeta_{ii}$, $i=1,2$.
\end{Lemma}

\begin{proof}
We set $$y_\ell(z,t):= \zeta( |z|^2,\varphi(\frac{2\pi t}{T}))z_\ell, \quad \ell=1,\cdots,N-1.$$
Then, on $\Omega_{T,\phi}$, we find  after computation,  
\begin{align}\label{eqpartiay}
\frac{\partial y_\ell }{\partial z_i }&=2 z_i  \zeta_1 z_\ell+ \zeta \delta_{i\ell}\nonumber\\
\frac{\partial^{2} y_\ell }{\partial z^{2}_i}&= 2  \frac{\zeta_1}{\zeta} y_\ell+4 y^{2}_i \frac{\zeta_{11}}{\zeta^3}  y_\ell+ 4 \frac{\zeta_1}{\zeta}   \delta_{i\ell} y_i,
\end{align}
where, here and in the following, we simply write $\zeta$ in place of the function
\begin{equation}
  \label{eq:zeta-replaced-z-t}
(z,t) \mapsto \zeta( |z|^2,\varphi(\frac{2\pi t}{T}))
\end{equation}
and similarly for $\zeta_i$, $\zeta_{ii}$, $i=1,2$. We  also have  $$ \frac{\partial  }{\partial t } \Bigl(\frac{\zeta_{2}(\tau) }{\zeta(\tau)}\Bigl)= \Bigl(\frac{2\pi}{T}\Bigl) \frac{\partial  }{\partial \t } \Bigl(\frac{\zeta_{2}(\tau) }{\zeta(\tau)}\Bigl)=\Bigl(\frac{2\pi}{T}\Bigl) \frac{\varphi'}{\zeta}\Bigl( \zeta_{22}-\frac{\zeta^2_{2}}{\zeta}\Bigl)$$
and hence
\begin{align}\label{eqpartiat}
\frac{\partial y_\ell }{\partial t }&= \frac{2\pi}{T}\varphi' \zeta_{2}z_\ell=\frac{2\pi}{T}\varphi' \frac{\zeta_{2}}{\zeta} y_\ell\nonumber\\
\frac{\partial^2 y_\ell }{\partial t^2 }&= \Bigl(\frac{2\pi}{T}\Bigl)^2 \Bigl(\varphi'' \frac{\zeta_{2}}{\zeta}  +\varphi'  \frac{\partial  }{\partial \t } \Bigl(\frac{\zeta_{2}(\tau) }{\zeta(\tau)}\Bigl)+ \varphi'^2 \frac{\zeta^2_{2} }{\zeta^2}\Bigl)y_\ell= \Bigl(\frac{2\pi}{T}\Bigl)^2 \Bigl( \frac{ \varphi'' \zeta_{2}+ \varphi'^2 \zeta_{22} }{\zeta}   \Bigl)y_\ell
\end{align}

Next we compute, on $\Omega_{T,\phi}$,
\begin{align*}
\frac{\partial u }{\partial z_i }&=\sum^{N-1}_{\ell=1} \frac{\partial y_\ell }{\partial z_i}
\Bigl(\frac{\partial w}{\partial y_\ell} \circ \Psi_{T,\phi}^{-1} \Bigr) \nonumber\\
\frac{\partial^{2}  u }{\partial z^{2}_i}&=  \sum^{N-1}_{\ell=1} \frac{\partial^2 y_\ell }{\partial z^2_i}\Bigl(
\frac{\partial w}{\partial y_\ell} \circ \Psi_{T,\phi}^{-1} \Bigr)  +\sum^{N-1}_{k, \ell=1} \frac{\partial y_\ell }{\partial z_i} \frac{\partial y_k }{\partial z_i} \Bigl(\frac{\partial^2  w}{\partial y_\ell \partial y_k} \circ \Psi_{T,\phi}^{-1} \Bigr)=:(A)+(B).
\end{align*}

Using  \eqref{eqpartiay}, we find
\begin{align*}
(A):&= \Bigl(2\frac{\zeta_1}{\zeta}+ 4 \frac{\zeta_{11} }{\zeta^3}y_{i}^2\Bigl)y\cdot \Bigl( \nabla_y w \circ \Psi_{T,\phi}^{-1} \Bigr)  +4\frac{\zeta_1}{\zeta} y_{i}\Bigl(\frac{\partial  w}{\partial y_i} \circ \Psi_{T,\phi}^{-1} \Bigr) \nonumber \\
(B):&=  \zeta^2  \Bigl( \frac{\partial^2  w}{ \partial y_{i}^{2}} \circ \Psi_{T,\phi}^{-1} \Bigr)+   4 \frac{\zeta_{1} }{\zeta}\sum^{N-1}_{k=1}  y_{i} y_k  \Bigl(\frac{\partial^2  w}{\partial y_i \partial y_k} \circ \Psi_{T,\phi}^{-1} \Bigr) +  4 y_{i}^2  \frac{\zeta_{1}^{2}}{\zeta^4}  \sum^{N-1}_{k, \ell=1}  y_\ell  y_k  \Bigl( \frac{\partial^2  w}{\partial y_\ell \partial y_k}\circ \Psi_{T,\phi}^{-1} \Bigr).
\end{align*}

In addition, \begin{align*}
\frac{\partial u }{\partial t }&=\sum^{N-1}_{\ell=1} \frac{\partial y_\ell }{\partial t}
\Bigl(\frac{\partial w}{\partial y_\ell} \circ \Psi_{T,\phi}^{-1} \Bigr) +\frac{2\pi}{T} \Bigl(\frac{\partial w}{ \partial \tau}\circ \Psi_{T,\phi}^{-1} \Bigr) \nonumber\\
\frac{\partial^{2}  u }{\partial t^{2}}&=  \sum^{N-1}_{\ell=1} \frac{\partial^2 y_\ell }{\partial t^2}
\Bigl(\frac{\partial w}{\partial y_\ell}\circ \Psi_{T,\phi}^{-1} \Bigr) +\sum^{N-1}_{k, \ell=1} \frac{\partial y_\ell }{\partial t} \frac{\partial y_k }{\partial t} \Bigl(\frac{\partial^2  w}{\partial y_\ell \partial y_k} \circ \Psi_{T,\phi}^{-1} \Bigr)+  \frac{2\pi}{T}\sum^{N-1}_{\ell=1} \frac{\partial y_\ell }{\partial t}
                                         \Bigl(\frac{\partial^2 w}{ \partial \tau\partial y_\ell} \circ \Psi_{T,\phi}^{-1} \Bigr)\\
               &+\left(\frac{2\pi}{T}\right)^2\Bigl( \frac{\partial^2 w}{ \partial \tau^2} \circ \Psi_{T,\phi}^{-1} \Bigr)\nonumber\\
&  =:(I)+(J)+(K).
\end{align*}
We now use \eqref{eqpartiat} and find
\begin{align*}
(I)&=  \Bigl(\frac{2\pi}{T}\Bigl)^2 \Bigl( \frac{ \varphi'' \zeta_{2}+ \varphi'^2 \zeta_{22} }{\zeta}   \Bigl)  y \cdot \Bigl(\nabla_y w \circ \Psi_{T,\phi}^{-1} \Bigr) \nonumber\\
(J)&= \Bigl(\frac{2\pi}{T}\Bigl)^2   \frac{\zeta^2_{2} \varphi'^2 }{\zeta^2}  \sum^{N-1}_{k,\ell=1}  y_k  y_\ell \Bigl( \frac{\partial^2  w}{\partial y_k y_\ell} \circ \Psi_{T,\phi}^{-1} \Bigr)\\
(K)&= \frac{2\pi}{T}\sum^{N-1}_{\ell=1} \frac{\partial y_\ell }{\partial t}
     \Bigl(\frac{\partial^2 w}{ \partial \tau\partial y_\ell} \circ \Psi_{T,\phi}^{-1} \Bigr) +\left(\frac{2\pi}{T}\right)^2\Bigl( \frac{\partial^2 w}{ \partial \tau^2} \circ \Psi_{T,\phi}^{-1} \Bigr)\\
  &=\left(\frac{2\pi}{T}\right)^2 \left[ \frac{\zeta_{2} \varphi' }{\zeta}  \sum^{N-1}_{\ell=1}  y_\ell \Bigl(\frac{\partial^2  w}{\partial y_\ell \partial \tau}\circ \Psi_{T,\phi}^{-1} \Bigr) +\Bigl(\frac{\partial^2 w}{ \partial \tau^2}\circ \Psi_{T,\phi}^{-1} \Bigr) \right]\nonumber.
\end{align*}
Collecting these identities, which we have derived on the domain $\Omega_{T,\phi}$ in the variables $(z,t)$, and passing to the variables $(y,\tau) \in \Omega_1$, we obtain the claim. Note here that we have to write $(z,t)= \bigl(\kappa ( |y|^2,  \varphi(\tau))y, \frac{T}{2\pi}\tau\bigr)$ to pass from (\ref{eq:zeta-replaced-z-t}) to (\ref{eq:zeta-replaced-y-tau}) and similarly for the partial derivatives $\zeta_i, \zeta_{ii}$, $i=1,2$.
\QED
\end{proof}

We now use Lemma \ref{eq:pulbal} to rephrase the original problem  \eqref{eq:ovder3} in the fixed domain
$ \B^c_{1} \times \R $. We recall that the  boundary  $\partial \O_{T, \varphi}$ of   $ \O_{T, \varphi}$  is given by
\begin{align}\label{boudary}
\partial \O_{T, \varphi} =\{(1+\varphi(\t))\sigma,\frac{T}{2\pi}\tau): \sigma\in\mathbb{S}^{N-2},t\in\mathbb{R}\}
\end{align}
 and  its  outer normal vector field  is given by
 \begin{align}\label{eq:outer2}
\Upsilon((1+\varphi(\t))\sigma,\frac{T}{2\pi}\tau) =\frac{\Bigl(-\sigma,  2 \pi \varphi'(\t)/T  \Bigl)}{\sqrt{1+ \Bigl(\frac{2\pi}{T} \Bigl)^2 \varphi'^2(\t)}}.
 \end{align}
Let the metric $g_{T,\varphi}$ be  defined as the pull back of the euclidean metric $g_{eucl}$ under the map $\Psi_{T,\varphi}$, so that $\Psi_{T,\varphi}:(\overline \O_{1} ,  g_{T,\varphi} ) \to (\overline   {\Omega_{T, \varphi}},g_{eucl})$ is an isometry.
Denote  by
$$
\eta_{T,\varphi}:  \partial \O_1 \to  \R^{N}
$$ the unit outer normal vector field on $\partial  \Omega_1 $ with respect
to $g_{T,\varphi} $.
Since  $\Psi_{T,\varphi}:(\overline \O_{1},  g_{T,\varphi} ) \to (\overline   {\Omega_{T, \varphi}} ,g_{eucl})$ is an isometry,  we have
\begin{equation}
  \label{eq:rel-mu-phi-nu-phi}
\eta_{T,\varphi} = [d \Psi_{T,\varphi}]^{-1}  \Upsilon  \circ \Psi_{T,\varphi}  \qquad
\text{on $\partial  \Omega_1 $.}
\end{equation}

Moreover, by (\ref{eq:rel-mu-phi-nu-phi}) we have
$\mu_\varphi(\Psi_\varphi (y,\tau) )=d\Psi_\varphi (y, \tau) \eta_\varphi (y, \tau) $ and therefore
\begin{align*}
\partial_{\eta_{T,\varphi}} w (y,\tau) &= d w (y, \tau) \eta_{\varphi} (y,\tau) =d u (\Psi_{\varphi}(y,\tau))  d \Psi_{\varphi}(y,\tau)  \eta_{T,\varphi} (y, \tau)\\
&= d u (\Psi_{\varphi}(y, \tau))  \Upsilon (\Psi_\varphi (y, \tau) )
  = \langle  \Upsilon (\Psi_\varphi (y, \tau) ), (\nabla_{ (z, t)}u) (\Psi_{\varphi}(y, \tau)) \rangle_{g_{eucl}}.
\end{align*}
That is
\begin{align}\label{norde11bac}
\partial_{\eta_{T,\varphi}} w (y,\tau)  &= \langle   \Upsilon (\Psi_\varphi (y, \tau) ), (\nabla_{ (z, t)}u) (\Psi_{\varphi}(y, \tau)) \rangle_{g_{eucl}}
\end{align}
where $u$ and $w$ are related by \eqref{eqans}.

Taking into  account \eqref{boudary}, our aim is to show that  for some values of the parameter $T>0$ and  $\varphi>-1$, we can find a solution $w$ to the overdetermined boundary value problem

\begin{align}\label{eq:ovder2rephrased}
\begin{cases}
L_{T, \varphi}w = 0& \quad  \textrm { in } \quad   \B^c_{1} \times \R  \vspace{3mm} \\
 w =  1& \quad\textrm{ on } \quad\partial (\B^c_{1} \times \R )\vspace{3mm}\\
 \lim \limits_{|y| \rightarrow \infty}  w (y, \t ) = 0& \quad \textrm{uniformly in $\t \in \R$}
\end{cases}
 \end{align}
 and
 \begin{align}\label{eq:overcondi22}
 \dfrac{\partial  w } {\partial \eta_{T,\varphi}}= N-3,&\quad
\textrm { on } \quad  \partial (\B^c_{1} \times \R ).
 \end{align}

We start proving the existence of solution to problem \eqref{eq:ovder2rephrased} by analysing the operator  $L_{T, \varphi}$ in    weighted  H\"{o}lder  spaces.

\section{Analysis of the operator $L_{T, \varphi}$ on   weighted H\"{o}lder  spaces } \label{sec:dirichlet}

Our  setting  in analysing the operator  $L_{T, \varphi}$ is that of  the  weighted H\"{o}lder  spaces introduced  by Pacard and Rivi\`ere \cite{PR}. We emphasise that  Pacard and Rivi\`ere  performed their analysis  on  $\B_1\backslash \{ 0 \}$  whereas in contrast we wish  to carry our study on the open set   $\Omega_1=\B^c_{1} \times \R $. More generally, we consider, for $r > 0$, the
sets
$$
\Omega_{r}:=\Big\{(y, \t )\in  \R^{N-1} \times \R:  |y|> r\}\subset\R^N,
$$
and we set $\Omega_0=\R^N$. Let $\alpha \in (0,1)$ be fixed in the following.

\begin{Definition}
\label{main-definition-hoelder-etc}
Let $\mu <0$, $k \in \mathbb{N}$.
\begin{enumerate}
\item[(i)] For a set $K \subset \R^N$ and a function $v \in C^{0,\alpha}(K)$ we put
  $$
  [v]_{C^{0,\alpha}(K)}:= \sup_{y,y' \in K}\frac{|v(y) - v(y')|}{|y - y'|^{\alpha}}
  $$
  and
  $$
  \|v\|_{C^{0,\alpha}(K)}:= \|v\|_{L^\infty(K)}+[v]_{C^{0,\alpha}(K)}.
  $$
\item[(ii)] We say $u \in
\mathcal{C}^{k,\alpha}_\mu( \ov \Omega_r)$ if $u \in
\mathcal{C}^{k,\alpha}_{\textrm{loc}}(\ov \Omega_r)$ and
$$\| u \|_{k,\alpha,\mu} = \sup_{s > r} \left ( s^{-\mu} [ u ]_{k,\alpha, s}  \right ) < \infty,$$
where $A_s = \{ (y, \tau)\in \R^{N-1}\times \R : \quad  s  \leqslant |y| \leqslant 2s\}$ \quad for $s >0 $ and
$$ [ u ]_{k,\alpha, s} := \sum_{i = 0}^k s^i \|\nabla^i u\|_{L^\infty(A_s)}   +s^{k+\alpha} [\nabla^k u]_{C^\alpha(A_s)}.
$$
\item[(iii)] We also define the following function spaces:
  \begin{align*}
   \mathcal{C}^{k,\alpha}_{\mu, \mathcal{D}}( \O_r)
&:= \left \{ u \in \mathcal{C}^{k,\alpha}_\mu(\O_r):  \quad  u|_{\partial \O_r} = 0 \right \}\\
\mathcal{C}^{k,\alpha}_{\mu,p, e}( \O_r)
    &:= \left \{ u \in  \mathcal{C}^{k,\alpha}_{\mu}( \O_r): \quad  u \textrm{ is  $2\pi$-periodic and even in the coordinate $t$}\right \},\\
\mathcal{C}^{k,\alpha}_{\mu, \mathcal{D}, p, e}( \O_r)
    &:= \left \{ u \in  \mathcal{C}^{k,\alpha}_{\mu, \mathcal{D}}( \O_r): \quad  u \textrm{ is  $2\pi$-periodic and even in the coordinate $t$}\right \}.
  \end{align*}
\end{enumerate}
\end{Definition}

\begin{Remark}
  \label{quick-est}
Let $r \ge 0$ and $u \in  \mathcal{C}^{k,\alpha}_{\mu}( \O_r)$. By definition, we then have
$$
\sup_{A_s} | u| \leqslant   [ u ]_{k,\alpha, s}\leqslant  s^\mu  \| u \|_{k,\alpha,\mu} \qquad \text{for all $s > r$}
$$
and therefore, in particular,
\begin{align}\label{eq estiA}
| u (y, \t)| \leqslant | y|^\mu  \| u \|_{k,\alpha,\mu} \qquad \text{for all $(y,\t) \in \Omega_r$}.
\end{align}
Consequently,
\begin{equation}
  \label{eq:asymptotic-decay}
|u (y, \t)| \to 0  \qquad \text{as $|y| \to \infty$ uniformly in $\t \in \R$ if $\mu < 0$.}
\end{equation}
\end{Remark}

It can be deduced from the specific form of the operator $L_{T, \varphi}$ given in Lemma  \ref{eq:pulbal} that $L_{T, \varphi}$ maps $\mathcal{C}^{2,\alpha}_{\mu, \mathcal{D}, p, e}( \O_1)$ to  $\mathcal{C}^{0,\alpha}_{\mu-2, p, e}( \O_1)$. The following is the main result of this section. Here and in the following, for two Banach spaces $X$  and $Y$, we let $\mathcal{L}(X, Y)$ denote the space of bounded linear operators from $X$ to $Y$, and $\Iso(X, Y)$ the subset of topological isomorphisms $X \to Y$.

 \begin{Theorem}\label{Oper-inv}
   Let $3-N < \mu < 0$. Then we have a smooth map
   \begin{equation}
     \label{eq:smooth-map}
   (0, +\infty )\times  \mathcal{U}\to
 \mathcal{L}( \mathcal{C}^{2,\alpha}_{\mu, p, e }(\O_1), \mathcal{C}^{0,\alpha}_{\mu-2, p, e }(\O_1)),\qquad (T,\varphi) \to L_{T, \varphi}.
   \end{equation}
   Moreover, there exists an open neighborhood $\cO \subset  (0, +\infty )\times  \mathcal{U}$ of $(0,\infty) \times \{0\}$ with the property that
\begin{equation}
  \label{eq:Iso-property}
L_{T, \varphi}^D:= L_{T, \varphi}\Big|_{\mathcal{C}^{k,\alpha}_{\mu, \mathcal{D}, p, e}(\O_1)}   \: \in \: \Iso(\mathcal{C}^{2,\alpha}_{\mu, \mathcal{D}, p, e}(\O_1), \mathcal{C}^{0,\alpha}_{\mu-2, p, e }(\O_1))
\qquad \text{for $(T,\varphi) \in \cO$.}
\end{equation}
\end{Theorem}

The main ingredient in the proof of Theorem~\ref{TheoinvertLaplace} is the following proposition.
\begin{Proposition}\label{TheoinvertLaplace}
Let $3-N <\mu < 0$ and $T>0$. Then the operator
$$
L_{T, 0}^D = \Delta_y + \Bigl(\frac{2\pi}{T}\Bigr)^2 \frac{\partial^2}{\partial \tau^2} : \ \mathcal{C}^{2,\alpha}_{\mu, \mathcal{D}, p, e}( \O_1)  \rightarrow
\mathcal{C}^{0,\alpha}_{\mu-2, p, e}( \O_1)
$$
is a topological isomorphism.
\end{Proposition}

Let us postpone the proof of Proposition~\ref{TheoinvertLaplace} for a moment and first quickly finish the proof of Theorem~\ref{Oper-inv}. Since
$$
\zeta_1(a,b) = O(a^{-2}),\quad \zeta_{11}(a,b) = O(a^{-3}), \quad \zeta_2(a,b) = O(a^{-1})\quad \text{and}\quad \zeta_{11}(a,b) = O(a^{-2})
$$
as $a \to +  \infty$, it follows by a straightforward computation from Lemma  \ref{eq:pulbal} that (\ref{eq:smooth-map}) defines a smooth map. Moreover, since, for any Banach spaces $X,Y$, the set $\Iso(X,Y)$ is open in $\cL(X,Y)$, it follows directly from Proposition~\ref{TheoinvertLaplace} and the continuity of the map $(T,\varphi) \to L_{T,\varphi}$ that there exists an open neighborhood $\cO \subset  (0, +\infty )\times  \mathcal{U}$ of $(0,\infty) \times \{0\}$ with the property that (\ref{eq:Iso-property}) holds.

So the proof of Theorem~\ref{Oper-inv} will be completed by proving Proposition~\ref{TheoinvertLaplace}, and this will be done in the remainder of this section. Without loss of generality, we may restrict our attention to the special case $T= 2\pi$,
in which we have
$$
L_{2\pi, 0} = \Delta,
$$
where $\Delta$ denotes the Laplace operator in the variables $(y,\tau) \in \R^N$. The general case will then follow by rescaling the $\tau$-variable. This will change the period length in the spaces $\mathcal{C}^{2,\alpha}_{\mu, \mathcal{D}, p, e}( \O_1)$ and $\mathcal{C}^{0,\alpha}_{\mu-2, p, e}( \O_1)$ but does not require further changes as the arguments below do not depend on the period length.

We first note the following.

\begin{Lemma}\label{injectivity}
Let $\mu < 0$. Then the operator
$$
\Delta : \ \mathcal{C}^{2,\alpha}_{\mu, \mathcal{D}, p, e}( \O_1)  \rightarrow
\mathcal{C}^{0,\alpha}_{\mu-2, p, e}( \O_1)
$$
is injective.
\end{Lemma}

\begin{proof}
  Let $w \in \mathcal{C}^{2,\alpha}_{\mu, \mathcal{D}, p, e}( \O_1)$. Then $w = 0$ on $\partial \O_1$ and
  $w(y,\t) \to 0$ as $|y| \to \infty$ uniformly in $\t$ by (\ref{eq:asymptotic-decay}). Since $w$ is continuous and also periodic in the $\tau$-variable, $w$ attains its maximum and minimum on $\overline \O_1$. Moreover, if $\Delta w=0$, then neither the maximum nor the mimimum can be attained in $\Omega_1$ unless $w$ is constant. In any case, we therefore conclude that $\Delta w = 0$ implies $w=0$, and thus the the lemma is proved. \QED
\end{proof}

As a consequence of the open mapping theorem, the proof of Proposition~\ref{TheoinvertLaplace} is completed once we have shown that
\begin{equation}
  \label{eq:surjectivity}
\Delta : \ \mathcal{C}^{2,\alpha}_{\mu, \mathcal{D}, p, e}( \O_1)  \rightarrow
\mathcal{C}^{0,\alpha}_{\mu-2, p, e}( \O_1) \quad \text{is surjective for $3-N <\mu < 0$.}
\end{equation}

To prove this, we let $3-N <\mu < 0$ and $f\in  \mathcal{C}^{0,\alpha}_{\mu-2, p, e}( \O_1)$ be fixed in the following. We are looking for a function $w\in \mathcal{C}^{2,\alpha}_{\mu, \mathcal{D}, p, e}( \O_1) $ such  that
$\Delta w=f$. We shall find this function in the form $w= w_1-w_2$, where
$$
w_1:= \Phi * \tilde f : \overline \O_1 \to \R
$$
where $\tilde f \in \mathcal{C}^{0,\alpha}_{\mu-2, p, e}(\R^N)$ is an arbitrary $\tau$-periodic, even and H\"{o}lder continuous extension of $f$ to $\R^N$ and $w_2$ is a $\tau$-periodic and even solution of
$$
\Delta w_2 =0 \quad \text{in $\Omega_1$},\qquad w_2 = w_1 \quad \text{on $\partial \Omega_1$.}
$$
Here
$$
x \mapsto \Phi(x)= c_N |x|^{2-N} \quad \text{is the fundamental solution associated with $-\Delta$ in $\R^N$,}
$$
 where
  $c_N= \frac{1}{(N-2)|S^{N-1}|}$. The surjectivity is therefore a consequence of the following two lemmas.

\begin{Lemma}
\label{decay-est}
Let $3-N < \mu < 0$ and $\tilde f \in \mathcal{C}^{0,\alpha}_{\mu-2, p, e}(\R^N)$ be as above.
Then
$$
\frac{1}{|\cdot|^{N-2}}* \tilde f \in \mathcal{C}^{2,\alpha}_{\mu,p,e}( \R^N).
$$
and therefore
$$
\Bigl(\frac{1}{|\cdot|^{N-2}}* \tilde f\Bigr)\Big|_{\overline \O_1} \in \mathcal{C}^{2,\alpha}_{\mu,p,e}( \O_1).
$$
\end{Lemma}

\begin{proof}
 By \eqref{eq estiA}) we have
  \begin{equation}
    \label{eq:assumption-1}
|f(y,\sigma)| \le |y|^{\mu-2} \|f\|_{0,\alpha,\mu-2} \qquad \text{for $y \in \R^{N-1}$, $\sigma \in \R$.}
\end{equation}
For $x \in \R^{N-1}\setminus \{0\}$ and $t \in \R$ we then find, by a change of variable, that
\begin{align}\label{ene1}
&\Bigl|\Bigl (\frac{1}{|\cdot|^{N-2}}* f\Bigr)(x,t) \Bigr| \le \int_{\R^{N-1}} \int_{\R} \Bigl(|x-y|^2+(t-\sigma)^2\Bigr)^{\frac{2-N}{2}}|f (y, \sigma) | \,d\sigma dy\\
&\le  \|f\|_{0,\alpha,\mu-2} \int_{\R^{N-1}}|y|^{\mu-2} |x-y|^{2-N}  \int_{\R} \Bigl(1+\Bigl(\frac{t-\sigma}{|x-y|}\Bigr)^2\Bigr)^{\frac{2-N}{2}}\,d\sigma dy\nonumber\\
&= C_N  \|f\|_{0,\alpha,\mu-2} \int_{\R^{N-1}}|y|^{\mu-2} |x-y|^{3-N} dy \quad \text{with}\quad C_N = \int_{\R} \Bigl(1+\tau^2\Bigr)^{\frac{2-N}{2}}\,d\tau < \infty.  \nonumber
\end{align}
By rotational invariance, we thus find that
\begin{align}
\Bigl|\Bigl(\frac{1}{|\cdot|^{N-2}}* f\Bigr)(x,t) \Bigr|&\le C_N \|f\|_{0,\alpha,\mu-2} \int_{\R^{N-1}}|y|^{\mu-2} \bigl| |x|e_1 -y\bigr |^{3-N} dy\nonumber\\
  &= C_N \|f\|_{0,\alpha,\mu-2} |x|^{\mu + 1-N} \int_{\R^{N-1}}\bigl|\frac{y}{|x|} \bigr |^{\mu-2} \bigl| e_1 -\frac{y}{|x|}\bigr |^{3-N} dy \nonumber\\
  &= C_N D_{N,\mu} \|f\|_{0,\alpha,\mu-2} |x|^{\mu}, \label{ene2}
\end{align}
where $D_{N,\mu}:=\int_{\R^{N-1}} |z|^{\mu-2} |e_1 -z|^{3-N} dz$ is finite since $3-N < \mu < 0$. Therefore, the function  $u:= \frac{1}{|\cdot|^{N-2}}* f$  satisfies the estimate
\begin{equation}
  \label{eq:Linfu1}
\|u\|_{L^\infty(A_s)}\le C_{N, \mu} \|f\|_{0,\alpha,\mu-2} s^{\mu}
\end{equation}
for some constant $C_{N, \mu}>0.$

Moreover, by  \eqref{eq:R-dependent} in the appendix, using the fact that $u$ solves $-\Delta u= f$ in $A_{2s} \cup A_{s}\cup A_{\frac{s}{2}} $, we find that for every $s > 0$,
\begin{align*}
&  \sum_{i = 0}^2 s^i \|\nabla^i  {u} \|_{A_s} +  s^{2+\alpha}[\nabla^2  {u} ]_{C^{0,\alpha} (A_s)}\\ &\le C \Bigl( s^2 \|f \|_{L^\infty( A_{2s} \cup A_{s} \cup A_{\frac{s}{2} })} + s^{2+\alpha}[f]_{C^{0,\alpha}(A_{2s} \cup A_{s}\cup A_{\frac{s}{2}})}
      + \|u \|_{L^\infty(A_{2s} \cup A_{s}\cup A_{\frac{s}{2}})}\Bigr)\\
    &\le C s^\mu \Bigl( \|f \|_{0,\alpha,\mu-2}
      + \|u\|_{0,\alpha,\mu}\Bigr)\le  C s^\mu \|f \|_{0,\alpha,\mu-2},
\end{align*}
where we used   \eqref{eq:Linfu1} in the last step. Hence $u \in \mathcal{C}^{2,\alpha}_{\mu}( \R^N)$. Moreover, $u$ is even and periodic in the $t$-variable, and therefore $u \in \mathcal{C}^{2,\alpha}_{\mu,p,e}( \R^N)$, as claimed. \QED
\end{proof}

\begin{Lemma}
  \label{harmonic-estimate}
  Let $3-N < \mu < 0$ and let $\phi \in C^{2,\alpha}_{p,e}(\ov \Omega_1)$.
  Then there exists $w \in \mathcal{C}^{2,\alpha}_{\mu,p,e}(\overline{\Omega_1})$ satisfying
  \begin{equation}
    \label{eq:harmonic-estimate-eq1}
  \Delta w= 0 \quad \text{in $\Omega_1$,} \qquad w = \phi \quad \text{on $\partial \Omega_1$.}
  \end{equation}
\end{Lemma}

\begin{proof}
  With the help of Perron's method, it is easy to see that (\ref{eq:harmonic-estimate-eq1}) admits an even and periodic solution $w \in \mathcal{C}^{2,\alpha}_{loc}(\overline{\Omega_1})$ with respect to $\tau$
   satisfying
  $$
  w(y,\tau) \to 0 \qquad \text{as $|y| \to \infty$ uniformly in $\tau \in \R$.}
  $$
  To see that $w \in \mathcal{C}^{2,\alpha}_{\mu,p,e}(\overline{\Omega_1})$, we first note that
  \begin{equation}
    \label{eq:harmonic-estimate-eq2}
    |w(y,\tau)| \le \|\phi\|_{C^{2,\alpha}_{p,e}(\ov \Omega_1)} |y|^{3-N}\qquad \text{for $(y,\tau) \in \Omega_1$,}
  \end{equation}
  by comparison (see \cite[Theorem 3.3]{GilbardTru}) with the functions
  $$
  w_{\pm}: \overline{\Omega_1} \to \R, \qquad  w_{\pm}(y,\tau) = \pm \|\phi\|_{C^{2,\alpha}_{p,e}(\ov \Omega_1)}|y|^{3-N},
  $$
  which are harmonic in $\Omega_1$ and satisfy $ w_{-} \leq w \le w_{+} $ on $\partial \Omega_1$.

Moreover, for $s>  2$, there exists $R_0>1$ such that   $ B_{R_0 s}(x)\subset A_{2s} \cup A_{s}\cup A_{\frac{s}{2}}$ for any $x\in A_s$. Using this and  the fact that ${w}$ solves $-\Delta {w}= 0$ in $A_{2s} \cup A_{s}\cup A_{\frac{s}{2}}$, we may apply \eqref{eq:R-dependen2} in the appendix with $f\equiv 0$ to see that
\begin{align}\label{estimmm2}
\sum_{i = 0}^2 s^i \|\nabla^i  {w} \|_{A_s} +   s^{2+\alpha}[\nabla^2  {w} ]_{C^{0,\alpha} (A_s)}
 \le C  \|{w} \|_{L^\infty(A_{2s} \cup A_{s}\cup A_{\frac{s}{2}})} &\le C \|\phi\|_{L^\infty( \Omega_1)} s^{3-N} \nonumber\\
 &\le C \|\phi\|_{L^\infty( \Omega_1)} s^{\mu}
\end{align}
Here we used that $\mu > 3-N$.  Using also the fact that, by standard elliptic estimates we have
$w \in \mathcal{C}^{2,\alpha}(\overline{\Omega_1 \setminus \Omega_2})$
with $\|w\|_{\mathcal{C}^{2,\alpha}(\overline{\Omega_1 \setminus \Omega_2})} \le C \|\phi\|_{C^{2,\alpha}(\ov \Omega_1)}$, we see that  for $1<s\leq2$,
\begin{align} \label{estimmm3}
\sum_{i = 0}^2 s^i \|\nabla^i  {w} \|_{A_s} +   s^{2+\alpha}[\nabla^2  {w} ]_{C^{0,\alpha} (A_s)}
 &\le C \|w\|_{\mathcal{C}^{2,\alpha}(\overline{\Omega_1 \setminus \Omega_2})}   \nonumber\\
  &\le  C \|\phi\|_{C^{2,\alpha}(\ov \Omega_1)}  \le \frac{C}{2^\mu}  s^\mu \|\phi\|_{C^{2,\alpha}(\ov \Omega_1)}
\end{align}
where $C$ is a constant only depending on $\alpha$.

Combining \eqref{estimmm2} and \eqref{estimmm3}, we obtain
\begin{equation}
  \label{eq:mu-mu-2-est-4new}
\|w\|_{2,\alpha,\mu} \le C \|\phi\|_{C^{2,\alpha}(\ov \Omega_1)},
\end{equation}
as required. \QED
\end{proof}
\section{Reformulation of the problem as a nonlinear operator equation}
\label{sec:soln_construction}

In this section, we wish to use Theorem~\ref{Oper-inv} to formulate problem (\ref{eq:ovder2rephrased}),~(\ref{eq:overcondi22}) as a nonlinear operator equation. For this, we need the following Lemma. Throughout the remainder of the paper, we let $\cO$ be given as in Theorem~\ref{Oper-inv}.

\begin{Lemma}\label{lem:Oper-inv}
  Let $3-N < \mu < 0$. Then there exists a smooth map
\begin{equation}\label{eq:smoothsol}
\cO \to \mathcal{C}^{2,\alpha}_{\mu, p, e}(\O_1), \quad (T, \varphi)\mapsto  w_{T, \varphi}
\end{equation}
with the property that, for every $(T,\varphi) \in \cO$, the function $w_{T,\varphi}$ is the unique solution of the problem
\begin{align}\label{eq:SoluDirichlet}
\begin{cases}
L_{T, \varphi} w_{T, \varphi}  = 0& \quad  \textrm { in } \quad   \B^c_{1} \times \R  \vspace{3mm} \\
 w_{T, \varphi} =  1& \quad\textrm{ on } \quad\partial (\B^c_{1} \times \R )\vspace{3mm}\\
 \lim \limits_{|y| \rightarrow \infty}  w_{T, \varphi}  (y, \tau ) = 0& \quad \textrm{uniformly in $\tau \in \R$.}
\end{cases}
 \end{align}
Moreover, the functions $w_{T, \varphi}: \overline \Omega_1 \to \R$ and $\partial_{\eta_\phi} w_{T, \varphi}: \partial \O_1 \to \R$
are radially symmetric in the $y$-variable, and
\begin{equation}
\label{eq:uat1}
w_{T,0}(y,\tau)=|y|^{3-N}= u_1(y,\tau) \qquad \text{for every $T>0$.}
\end{equation}
\end{Lemma}

\proof
We first note that $u_1\in  \mathcal{C}^{2,\alpha}_{\mu}(\O_1)$  since  $3-N < \mu < 0$, as $|\partial^i u_1(y,\t)|\le (N-3)|y|^{2-N}$ and  $|\partial^{ij} u_1(y,\t)| \le N(N-3)|y|^{1-N}$ for $i,j = 1,\dots,N$. We consider the map
\begin{equation}
  \label{eq:def-u-t-phi}
\cO \to \mathcal{C}^{2,\alpha}_{\mu, \mathcal{D}, p, e}( \O_1), \qquad (T,\varphi) \mapsto m_{T, \varphi}:= \Bigl(L_{T,\varphi}^D\Bigr)^{-1}\Bigl(L_{T,\varphi} u_1\Bigr),
\end{equation}
which is well-defined by Theorem~\ref{Oper-inv}. For $(T,\varphi) \in \cO$, the function $m_{T, \varphi} \in \mathcal{C}^{2,\alpha}_{\mu, \mathcal{D}, p, e}( \O_1)$ is the unique solution of the problem
\begin{align}\label{eq:SoluDirichlet2}
\begin{cases}
L_{T, \varphi} m_{T, \varphi}  =  L_{T, \varphi} u_1& \quad  \textrm { in } \quad   \O_1  \vspace{3mm} \\
m_{T, \varphi} =  0& \quad\textrm{ on } \quad\partial  \O_1\\
\lim \limits_{|y| \rightarrow \infty}  m_{T, \varphi}  (y, \tau ) = 0& \quad \textrm{uniformly in $\tau \in \R$.}
\end{cases}
 \end{align}
Hence the function
\begin{equation}\label{eq:SoluDifinial}
w_{T, \varphi}: =-m_{T, \varphi}+ u_1 \; \in \; \mathcal{C}^{2,\alpha}_{\mu, p, e}(\O_1)
 \end{equation}
 is the unique solution of \eqref{eq:SoluDirichlet}. Moreover, both $w_{T, \varphi}: \overline \Omega_1 \to \R$ and $\partial_{\eta_\phi} w_{T, \varphi}: \partial \O_1 \to \R$ are radially symmetric in the $x$-variable by uniqueness and the fact that both $\Omega_1$ and the operator $L_{T,\varphi}$ are radial in the $x$-variable. In addition, for $T>0$ we have  $m_{T, 0} \equiv 0$ in $\O_1$ since $L_{T,0} u_1 = \Delta_y u_1 \equiv 0$ and therefore $w_{T,0}=u_1$.

 It thus remains to show that the map $(T,\varphi) \mapsto m_{T,\varphi}$ in (\ref{eq:def-u-t-phi}) is smooth. For this we first note that, for every pair of Banach spaces $X,Y$, the inversion map
 $$
 \mathcal{I}(X, Y) \to \mathcal{I}(Y, X),  \qquad T \mapsto T^{-1}
 $$
 is smooth in the open set $\mathcal{I}(X, Y) \subset \mathcal{L}(X, Y)$ of topological isomorphisms.
 Hence the smoothness of the maps $(T,\varphi) \mapsto m_{T,\varphi}$ follows by the smoothness of the maps
 \begin{align*}
   \cO \to  \mathcal{L}( \mathcal{C}^{2,\alpha}_{\mu, p, e }(\O_1), \mathcal{C}^{0,\alpha}_{\mu-2, p, e }(\O_1)),  \qquad (T,\varphi) \to L_{T, \varphi},\\
   \cO \to  \mathcal{L}(\mathcal{C}^{2,\alpha}_{\mu, \mathcal{D}, p, e}( \O_1), \mathcal{C}^{0,\alpha}_{\mu-2, p, e }(\O_1)),   \qquad (T,\varphi) \to L_{T, \varphi}^D
 \end{align*}
asserted in Theorem~\ref{Oper-inv}. The proof is thus finished.\QED

The aim now is to prove that for some  parameter values $(T,\varphi) \in \cO$ with $\varphi \not \equiv 0$, the function $w_{,T\varphi}$  satisfies the overdetermined condition
\begin{align}\label{eq:overcondi2}
\dfrac{\partial  w_{T,\varphi} } {\partial \eta_{T,\varphi}}= N-3 &\quad
\textrm { on } \quad  \partial (\B^c_{1} \times \R ).
 \end{align}
We thus define the map
\begin{equation}\label{CR-map}
  F: \cO  \subset \R \times    C^{2,\alpha}_{p,e}(\R)   \to   C^{1,\alpha}_{p,e}(\R), \quad
  F (T,\varphi)(\tau): = \dfrac{\partial  w_{T,\varphi} } {\partial \eta_{T,\varphi}}(e_1,\tau)-(N-3)
\end{equation}
By radial symmetry of $\dfrac{\partial  w_{T,\varphi} } {\partial \eta_{T,\varphi}}$, the condition~(\ref{eq:overcondi2}) is therefore equivalent to $F(T,\varphi)=0$. Our aim is to apply the Crandall-Rabinowitz bifurcation theorem\cite[Theorem 1.7]{M.CR} to  solve the equation
\begin{equation}\label{bifurcaequa}
F(T,\varphi)\equiv0  \quad\textrm{in}\quad C^{1, \alpha}_{p,e}(\R).
\end{equation}
Observe that  from \eqref{eq:outer2} and \eqref{eq:rel-mu-phi-nu-phi} that the map $\cO \to  C^{1, \alpha}(\partial \O_1, \R^N)$, $(T,\varphi)\mapsto \eta_{T,\varphi} $ is  smooth. This together with the smoothness of the map in   \eqref{eq:smoothsol} guarantees   that $(T,\varphi)\mapsto F(T,\varphi) $   is smooth. Indeed, we have the following observation.

\begin{Lemma} (see \cite[Lemma 2.3]{Fall-MinlendI-Weth})\\
Let
$$
h: \cO \to C^{2,\alpha}(\overline \O_1), \qquad (T,\phi) \mapsto h_{T,\phi}
$$
be a smooth map. Then the map
$$
\cG(T,\phi): \cO \to C^{1,\alpha}(\partial \O_1), \qquad \cG(\phi)=
 \frac{\partial h_{T,\phi}}{\partial \eta_{T,\phi}}
$$
is smooth as well and satisfies
\begin{align}\label{eq:cainenormader}
D_\phi \cG(T,\phi)v  = \frac{\partial h_{T,\phi}}{\partial \tilde \eta_{T,\phi} (v)}
+  \frac{\partial \Bigl([D_{\phi} h_{T,\phi}] v \Bigr)}{\partial \eta_{T,\phi}}  \qquad
\text{for $v\in C^{2,\alpha}_{p,e}(\R)$,}
 \end{align}
 where
 $$
\tilde \eta_{T,\phi}(v):= [D_\phi \eta_{T,\phi}]v \in
C^{1,\alpha}_{p,e}(\partial \O_1,\R^{N}) \qquad \text{for $(T,\phi) \in
\cO,\: v \in C^{2,\alpha}_{p,e}(\R)$.}
$$
\end{Lemma}

In addition,  by \eqref{norde11bac}, \eqref{bifurcaequa} reads
\begin{align}\label{eq:opF}
F(T, \varphi)(t)=\nabla u_{T,\varphi}( (1+\varphi(\t))e_1,\frac{T}{2\pi}\tau))\cdot \Upsilon ((1+\varphi(\t))e_1, \frac{T}{2\pi}\tau) -(N-3),
 \end{align}
where $u_{T,\varphi}$ and $w_{T,\varphi}$ are related by \eqref{eqans}. Hence applying \eqref{eq:opF} with $\varphi=0$ and using \eqref{eq:uat1},   we see that
 $$
F(T, 0)=0 \qquad \text{for every $T>0$ }.
 $$

\section{Study of the linearised operator}\label{linearizaton}
The aim of this section is to study the spectral properties of the linearized operator
\begin{equation}
  \label{eq:def-H}
H_T \in \cL(C^{2,\alpha}_{p,e}(\R),C^{1,\alpha}_{p,e}(\R)),\qquad H_T(v)= D_\varphi \big|_{\varphi=0} F(T, \varphi)v.
\end{equation}

\begin{Proposition}\label{outer2}
For every  $T>0$, the linearised operator $H_T$ defined by  \eqref{eq:def-H}   is given by
 \begin{align}\label{eqlinear0}
 H_T(v)(\t) &= (N-3)\Bigl(\nabla \dot{u}(e_1,\t) \cdot(-e_1,0)-(N-2)v(\t)\Bigl),
\end{align}
where
$\dot{u}$ is solution to
\begin{equation}\label{eqlidotu}
	\begin{cases}
	\Bigl( \Delta_y + \left(\frac{2\pi}{T}\right)^2\frac{\partial^2}{ \partial \tau^2}\Bigr) \dot{u} =0& \quad  \textrm { in } \quad   \B^c_{1} \times \R  \vspace{3mm} \\
	\dot{u}(y,\t) =v(\tau) & \quad  (y,\tau) \in \partial ( \B^c_{1} \times \R)  \vspace{2mm} \\
	\dot{u}\rightarrow 0 &\emph{as $|z|\rightarrow\infty$ uniformly in $\t\in\mathbb{R}.$}
	\end{cases}
	\end{equation}
Furthermore, the  eigenfunctions of  $H_T$ are given by  $v_k(\tau):= \cos(k\t)$, $k \in \N \cup \{0\}$, and we have
\begin{equation}\label{eq:eigenvalues}
H_Tv_k= \lambda_k(T)v_k,
\end{equation}
where
\begin{align}\label{eq:neweigen}
\lambda_k(T)=-(N-3)\Lambda(\frac{2 k\pi}{T})
 \end{align}
with $\Lambda(0)=1$ and
\begin{equation}\label{eqVtilde}
\Lambda(\rho)=\left(N-2-\frac{\rho K_{\nu+1}(\rho)}{K_{\nu}(\rho)}\right) \quad \text{for $\rho>0$ with $\nu=\frac{N-3}{2}$.}
\end{equation}	
Here and in the following, for $\nu \ge 0$, the function $K_\nu$ is the modified Bessel function of the second kind of order $\nu$.
\end{Proposition}

\proof
To prove \eqref{eqlinear0}   and   \eqref{eqlidotu},  we consider the functions
\begin{align}\label{eq: ChafU1}
W_{T, \varphi}(y, \tau):&= u_{1}( \kappa( |y|^{2}, \varphi(\tau))y, \frac{T}{2\pi}\tau) = |y|^{3-N} \Bigl(1+\frac{\varphi(\tau)}{|y|^{2}}\Bigl)^{3-N}\nonumber\\
V_{T, \varphi}(y, \tau):&= W_{T, \varphi}(y, \tau)-w_{T, \varphi}(y, \tau)
 \end{align}
Since  $u_{1}(y,\tau)=|y|^{3-N}$ is harmonic in $\R^N \setminus (\{0\} \times \R)$, we have, by \eqref{eqans} and \eqref{ebiject},
\begin{align}\label{eq:Harminicbac}
L_{T, \varphi}  W_{T, \varphi}  =0& \quad  \textrm { in } \quad   \O_1.
 \end{align}
 Moreover,
\begin{align}\label{eqWT}
\begin{cases}
L_{T, \varphi}  V_{T, \varphi}  = 0 & \quad  \textrm { in } \quad     \O_1    \vspace{3mm} \\
V_{T, \varphi}(y,\tau)  =(1+\varphi(\tau))^{3-N}-1& \quad\textrm{for $(y,\tau) \in \partial \O_1$}
\end{cases}
 \end{align}
 and
\begin{align}\label{eq: Wvanish}
V_{T, 0} \equiv0 \quad  \textrm { in } \quad     \ov \O_1
\end{align}
Differentiating  \eqref{eqWT} with respect to $\varphi$ at $\varphi \equiv0$ and setting
\begin{align}\label{eq:defww}
\psi_{T, v}:=\frac{1}{3-N} [D_\varphi \big|_{\varphi=0}   V_{T, \varphi}]v \qquad \text{for $v \in C^{2,\alpha}_{p,e}(\R)$,}
\end{align}
we find that
\begin{align}\label{eq:def-w1}
\begin{cases}
	\Bigl( \Delta_y + \left(\frac{2\pi}{T}\right)^2\frac{\partial^2}{ \partial \tau^2}\Bigr)   \psi_{T, v}  = 0&
\textrm { in } \quad\O_1 \vspace{3mm}\\
   \psi_{T, v}  =  v &
\textrm{ on }  \partial \O_1 \vspace{3mm}\\
 \psi_{T, v} \rightarrow 0\quad  \emph{as $|y|\rightarrow\infty$ uniformly in $\tau\in\mathbb{R}.$}
\end{cases}
\end{align}

We now put $G(T, \varphi)(\tau):= \dfrac{\partial   W_{T, \varphi}  } {\partial \eta_{T,\varphi}}(e_1,\tau)$ for $\tau \in \R$. By (\ref{eq:outer2}), \eqref{norde11bac} and \eqref{eq: ChafU1}, we have
\begin{align*}
  &G(T, \varphi)(\tau)=(\nabla u_{1})( (1+\varphi(\t))e_1,\t))\cdot \Upsilon( (1+\varphi(\t))e_1,\t)\\
  &= \Bigl(-(N-3)(1+\varphi(\t))^{2-N} e_1\Bigr) \cdot \Bigl( - \Bigl(1+ \bigl(\frac{2\pi}{T} \bigl)^2 \varphi'^2(\t)\Bigr)^{-\frac{1}{2}}e_1 \Bigr)\\
  &= (N-3)(1+\varphi(\t))^{2-N}\Bigl(1+ \bigl(\frac{2\pi}{T} \bigl)^2 \varphi'^2(\t)\Bigr)^{-\frac{1}{2}}.
 \end{align*}
Consequently,
\begin{align}\label{eqlinear2}
\Bigl[D_\varphi \big|_{\varphi=0}  G(T, \varphi)v\Bigr](\t) = - (N-2)(N-3) v(\tau).
\end{align}
Moreover, by  \eqref{eq:cainenormader}  and since $V_{T, 0} \equiv0$ in $\ov \O_1$, we have
\begin{equation}
  \label{eq:proof-linearization-14}
\Bigl[D_\phi \big|_{\phi=0}  \frac{\partial V_{T, \varphi}}{\partial \eta_{T,\varphi} }\Bigr]\omega =
\partial_{\tilde \eta_{\phi}(\omega)} V_{T, 0}  +
\partial_{\eta}  \psi_{T, v} =
-(N-3)\partial_\eta \psi_{T, v} \qquad \text{on $\partial
\Omega_1$},
\end{equation}
where $\eta$ is the outer unit normal on $\partial \O_1$ with respect
to $g_{eucl}$ given by $\eta(y,\tau)= (-z,0)$.
Using  \eqref{eq: ChafU1}, (\ref{eqlinear2}) and (\ref{eq:proof-linearization-14}), we get
$$
\Bigl[D_\varphi \big|_{\varphi=0} \frac{\partial w_{T,\phi}}{\partial \eta_{T,\phi}}\Bigr] v=  \Bigl[D_\varphi
\big|_{\phi=0} \frac{\partial}{\partial \eta_{T,\phi}} (W_{T,\phi} - V_{T,\phi})( \sigma,\cdot)\Bigr]v
=    (N-3)\Bigl(\partial_{\eta}  \psi_{T, v} - (N-2) v\Bigl)
$$
which combined with \eqref{CR-map} and (\ref{eq:def-H}) yields
$$
 H_T(v)(\t)= \Bigl[D_\varphi \big|_{\varphi=0}   F(T, \varphi)v\Bigr](\tau) = (N-3)\Bigl(\nabla \dot{u}(e_1,\t) \cdot(-e_1,0)-(N-2)v(\t)\Bigl),
$$
as claimed in (\ref{eqlinear0}). This proves the first part of Proposition \ref{outer2}.\\

Next we claim that  the functions $\tau \mapsto v_k(\t):= \cos(k\t)$, $k \in \N \cup \{0\}$ are the eigenfunctions of the operator
$v\mapsto  H_T(v).$ This is clear for $k=0$, as the unique solution of (\ref{eqlidotu}) with $v \equiv 1$ is merely given by $(y,\t) \to u_1(y,\tau)= |y|^{3-N}$, and therefore $H_T(v) \equiv -(N-3)$ by (\ref{eqlinear0}).

Assuming $k\neq0$ from now on, we see that the unique solution  $\dot{u}_k$ to \eqref{eqlidotu} with $v = v_k$ can be expressed, by separation of variables, as
$\dot{u}_k (y,\tau)=u_k(|y|) \cos(k\tau)$ where $u_k$ solves the ODE boundary value problem
\begin{equation}
	\begin{cases}
    u''_k +\dfrac{N-2}{r}u'_k -(\frac{2\pi k}{T})^2 u_k =0&\emph{in $(1,\infty),$}  \vspace{3mm}\\
u_k (1)=1,  \vspace{3mm}\\
u_k (r)\rightarrow 0 &\emph{as $r\rightarrow\infty$.}
	\end{cases}
	\end{equation}
We set $\rho_k:= \frac{2\pi k}{T}$ and consider  the function $w: [\rho_k,\infty)\rightarrow\mathbb{R}$ defined by
\begin{equation}\label{eqwuu}
w(\rho)=u_k(\frac{\rho}{\rho_k}).
\end{equation}
We have
\begin{equation*}
	\begin{cases}
 w''(\rho) +\frac{N-2}{\rho}w'(\rho) -w(\rho) =0&\emph{in $(\rho_k,\infty),$} \vspace{3mm}\\
w (\rho_k)=1,  \vspace{3mm}\\
w (\rho)\rightarrow0 &\emph{as $\rho\rightarrow\infty$,}
\end{cases}
\end{equation*}
and setting   $g(\rho):=\rho^{\nu}w(\rho)$, where  $\nu=\frac{N-3}{2}$,  we see that  function  $g$  satisfies
\begin{equation}\label{eqmodif1}
\begin{cases}
g(\rho)+\frac{1}{\rho}g'(\rho)-(1+\frac{\nu{2}}{\rho^2})g(\rho)=0&\emph{in $(\rho_k,\infty),$} \vspace{3mm}\\
g (\rho)\rightarrow0 &\emph{as $\rho\rightarrow\infty$,}\vspace{3mm}\\
g(\rho_k)=\rho_k^{\nu}.
\end{cases}
\end{equation}

Up to a multiplicative constant, the modified Bessel function of second kind  $K_\nu$ is the unique solution to  \eqref{eqmodif1}. Since  the function $K_\nu$ is positive on $(0, +\infty)$ it follows  that
\begin{equation}\label{eqmodif}
w(\rho)= c\rho^{-\nu} K_\nu(\rho), \quad \rho\in (\rho_k, +\infty), \quad\textrm{for some  constant $c > 0$}.
\end{equation}
Furthermore, combining  \eqref{eqwuu} with  \eqref{eqmodif},
\begin{equation}\label{eqmodiuk}
u_k(r)=w(K_0r)=c\rho_k^{-\nu}r^{-\nu}K_{\nu}(\rho_kr)
\end{equation}
and it follows from $u_k (1)=1$ that $c= \rho_k^{\nu}/ K_{\nu}(\rho_k)$
and
\begin{equation}\label{eqmderuu}
u'_k(1)=c\rho_k^{-\nu}K_{\nu}(\rho_k)\Bigl(-\nu+ \rho_k\frac{ K'_{\nu}(\rho_k)}{ K_{\nu}(\rho_k)}\Bigl)=-\nu+ \rho_k\frac{ K'_{\nu}(\rho_k)}{ K_{\nu}(\rho_k)}
\end{equation}

Recalling  \eqref{eqlinear0} and \eqref{eqmderuu}, we find that
 \begin{align*}
H_T (v_k)(\t)  &= (N-3)\Biggl(\nabla  u_k(e_1,\t)\cdot(-e_1,0)-(N-2)v_k(\t)\Bigl)\nonumber\\
  &=-(N-3)\left(\frac{N-1}{2}+ \rho_k\frac{ K'_{\nu}(\rho_k)}{ K_{\nu}(\rho_k)}\right)v_k(\t) \nonumber\\
  &=-(N-3)\left(N-2-\rho_k\frac{ K_{\nu+1}(\rho_k)}{ K_{\nu}(\rho_k)}\right)v_k(\t),
\end{align*}
where we have used the relation
\[\rho K'_{\nu}(\rho)=\nu K_{\nu}(\rho)-\rho K_{\nu+1}(\rho).\]
Consequently,
\begin{align}
H_T (v_k)  &=-(N-3)\Lambda(\rho_k) v_k  =-(N-3)\Lambda(\frac{2k\pi}{T}) v_k
\end{align}
with $\Lambda$ given in (\ref{eqVtilde}). Hence (\ref{eq:eigenvalues}) follows with $\lambda_k(T)$ given in (\ref{eq:neweigen}), and the proof is finished.
\QED

In the following result, we study the behaviour of the eigenvalue in  $\lambda_1(T)$ in   \eqref{eq:neweigen}.

\begin{Lemma}\label{dervtildeV}
For $k >0$ we have
\begin{equation} \label{eq:dasym1}
	\mu_k(T) \rightarrow
	\begin{cases}
	-(N-3), & T  \rightarrow +\infty ,\\
	+\infty, &T \rightarrow 0^+.
	\end{cases}
\end{equation}	
and for every $T>0$ we have
\be\label{eq:growth-eigen}
\lim \limits_{k\to\infty}   \frac{\lambda_k(T)}{k}  =\frac{2\pi (N-3)}{T}.
\ee
Moreover,  there exists a unique $T_*> \dfrac{2\pi}{\sqrt{N-2}}$ satisfying
\begin{equation}
\label{derv222-second-claim}
\text{ $\lambda_1(T_*)=0$, $\lambda'_1(T_*)< 0$ and $\lambda_k(T_*) \not = 0$ for $k \not = 1$.}
\end{equation}
\end{Lemma}

\proof
The proof of the above result is achieved studying the asymptotics for the function $\Lambda$ defined in \eqref{eqVtilde} with $\nu = \frac{N-3}{2}$. By \eqref{eqasymp1} and (\ref{eqasymp2}) in the appendix, we have
\begin{align*}
	\rho\frac{K_{\nu+1}(\rho)}{K_{\nu}(\rho)}\rightarrow
	\begin{cases}
	2\nu =N-3, &\rho  \rightarrow 0^+,\\
	+\infty, &\rho  \rightarrow +\infty.
	\end{cases}
\end{align*}
and therefore (\ref{eq:dasym1}) follows. Furthermore, since
\begin{align*}
	\frac{K_{\nu+1}(\rho)}{K_{\nu}(\rho)}\rightarrow 1  \qquad  \textrm{as}\quad \rho  \rightarrow +\infty
\end{align*}
by \eqref{eqasymp1}, we have
\[\lim \limits_{\rho\to\infty}\frac{\Lambda(\rho)}{\rho}=\lim \limits_{\rho\to\infty}\left(\frac{N-2}{\rho}-\frac{K_{\nu+1}(\rho)}{K_{\nu}(\rho)}\right)=-1\]
and hence
$$\lim \limits_{k\to\infty}   \frac{\lambda_k(T)}{k}= -\frac{2\pi(N-3)}{T}\lim \limits_{k\to\infty}\frac{T}{2\pi k}\Lambda(\frac{2\pi k}{T})=\frac{2\pi(N-3)}{T},$$
proving \eqref{eq:growth-eigen}.

To prove \eqref{derv222-second-claim}, we note that the function $T \mapsto \mu_1(T)$ has a positive zero by \ref{eq:growth-eigen} and (\ref{eq:dasym1}). We then use \eqref{relateK1} and \eqref{relateK2} to compute
\begin{align}\label{eq:deriv1}
-\Lambda'(\rho)&=\frac{K_{\nu+1}(\rho)}{K_{\upsilon}(\rho)}+\rho\frac{K'_{\nu+1}(\rho)}{K_{\nu}(\rho)}-\rho\frac{K'_{\nu}(\rho)}{K_{\nu}(\rho)}  \frac{K_{\nu+1}(\rho)}{K_{\nu}(\rho)},\nonumber\\
&=\frac{K_{\nu+1}(\rho)}{K_{\nu}(\rho)} -\rho -(\nu+1) \frac{K_{\nu+1}(\rho)}{K_{\nu}(\rho)}+ \Bigl( -\nu +\rho\frac{K_{\nu+1}(x)}{K_{\nu}(\rho)} \Bigl)  \frac{K_{\nu+1}(\rho)}{K_{\nu}(\rho)},\nonumber\\
&=\frac{K_{\nu+1}(\rho)}{K_{\nu}(\rho)}\Bigl(-2\nu  + \rho\frac{K_{\nu+1}(x)}{K_{\nu}(\rho)} \Bigl)-\rho \qquad \text{for $\rho>0$.}
\end{align}
Next, we consider a point $\rho>0$ with $\Lambda(\rho)=0$. Then
\begin{align}\label{eq:zero}
\frac{K_{\nu+1}(\rho)}{K_{\nu}(\rho)}=\frac{N-2}{\rho}
\end{align}
and, by \eqref{eq:compI1I00},
$$
\nu +\sqrt{\rho^2+\nu^2} < \rho\frac{K_{\nu+1}(\rho)}{ K_{\nu}(\rho)} =N-2   \Longleftrightarrow  \sqrt{\rho^2+\nu^2}<  -\nu + N-2  = \frac{N-1}{2}
$$
since $\nu = \frac{N-3}{2}$, which gives
\begin{align}\label{eq:inezer}
\rho < \sqrt{N-2}.
\end{align}
Plugging  \eqref{eq:zero} in  \eqref{eq:deriv1} and using \eqref{eq:inezer} yields
\begin{align}\label{eq:derineiv2}
-\Lambda'(\rho)&=\frac{N-2}{\rho}-\rho>0.
\end{align}
It thus follows that the function $\Lambda$ has a unique zero $\rho_*$ on $(0,\infty)$ satisfying $\rho* < \sqrt{N-2}$. Hence  \eqref{eq:neweigen} gives  \eqref{derv222-second-claim} with  $T_*> \dfrac{2\pi}{\sqrt{N-2}}$.
\QED

\section{Proof of the main result}
\label{sec:Main result}

In  this section we complete the proof of Theorem~\ref{Theo1}. For this we consider  the fractional Sobolev spaces
\begin{equation}
  \label{eq:def-hpe}
H^{\s}_{p,e} := \Bigl \{v \in H^{\s}_{loc}(\R) \::\: \text{
$v$ even and $2\pi$-periodic}\Bigl \}
\end{equation}
for $\s\geq 0$, and we put $L^2_{p,e}:= H^{0}_{p,e}$. Note that
$L^2_{p,e}$ is a Hilbert space with scalar product
$$
(u,v) \mapsto \langle u,v \rangle_{L^2_{p,e}} :=
\int_{-\pi}^{\pi} u(\tau)v(\tau)\,d\tau \qquad \text{for $u,v \in
L^2_{p,e}$.}
$$
and   induced norm  denoted by $\|\cdot\|_{L^2_{p,e}}$.  We define for all  $k \in \N \cup \{0\}$,  $v_k(t):=\cos(k\tau)$.
$\|v_k\|_{L^2_{p,e}}=\sqrt{\pi}$, the set
$\{\frac{v_k}{\sqrt{\pi}},\:\,k\in \N\}$ forms  a
complete orthonormal basis of $L^2_{p,e}$.
Moreover, $H^\s_{p,e}
\subset L^2_{p,e}$ is characterized as the subspace of all functions
$v \in L^2_{p,e}$ such that
\begin{align}\label{eqchatacte}
\sum_{k \in \N } (1+k^2)^{\s} \langle v, v_k
\rangle_{L^2_{p,e}} ^2 < \infty.
\end{align}
Therefore, $H^\s_{p,e}$ is also a Hilbert space with scalar product
\begin{equation}
  \label{eq:scp-hj}
(u,v) \mapsto  \sum_{k \in \N } (1+k^2)^{\s} \langle u, v_k
\rangle_{L^2_{p,e}}  \langle v, v_k \rangle_{L^2_{p,e}}  \qquad
\text{for $u,v \in H^\s_{p,e}$.}
\end{equation}
Set
\begin{equation}
  \label{eq:def-Vell}
W_k:=\textrm{span}\left\{ v_k \right\}
\,\subset \,\bigcap_{j \in \N} H^j_{p,e}
\end{equation}
for $k \in  \N$. Then from Proposition \ref{outer2},   the spaces $W_k$ are the eigenspaces of the operator $H_T$ in (\ref{eq:def-H}) corresponding to the eigenvalues $\lambda_k(T)$, i.e., we have
 \be\label{eq:-Lv-Fpurier-slab}
 H_T v =\lambda_k(T) v \qquad \textrm{ for every $v\in W_k$}.
\ee
We also consider their orthogonal complements in $L^2_{p,e}$, given by
$$
W_k^\perp:=\left\{ w\in L^2_{p,e} \,:\,
\int_{-\pi}^{\pi} \cos(k s) w(s)\,ds=0 \right\}
$$
as well as  the the spaces
$$
X:=\biggl\{  \phi: \mathbb{R} \rightarrow \mathbb{R},\quad  \phi\in  C^{2,\alpha}(\mathbb{R}) \textrm{ is even and $2\pi$-periodic}\biggl\},
$$
and
$$
Y:=\biggl\{  \phi: \mathbb{R} \rightarrow \mathbb{R},\quad  \phi\in  C^{1,\alpha}(\mathbb{R}) \textrm{ is even and $2\pi$-periodic} \biggl\}.
$$

\begin{Proposition}\label{prop:All-HYPcRANDAL}
 There exists a unique $T_*>0$ such that the linear operator $H_*:= H_{T_*}: X \to Y $  has the following properties.
\begin{itemize}
\item[(i)] The kernel $N(H_*)$ of $H_*$ is spanned by the function $\cos(\cdot).$
\item[(ii)] $H_* \big|_{X\cap W_1^\perp}: X\cap W_1^\perp  \to Y \cap W_1^\perp$ is an isomorphism.
\end{itemize}
Moreover
\begin{equation}
  \label{eq:transversality-cond}
\partial_T \Bigl|_{T= T_*} H_Tv_1  =\lambda_1'(T_*)v_1 \not  \in Y \cap W_1^\perp.
\end{equation}
 \end{Proposition}

\begin{proof} By Lemma \ref{dervtildeV}, we have the existence  of a unique $T_*>0$ such
that $\lambda_1(T_*)=0$, $ \mu'_1(T_*)<0$ and  $\lambda_k(T_*)\ne 0$  for all $k\ne1$.  This with \eqref{eq:-Lv-Fpurier-slab}  imply that
$N(H_*)=\textrm{span}\{v_1\} =W_1$ and we  obtain  (i) and \eqref{eq:transversality-cond}.\\

To prove (ii), we pick  $ g\in Y\cap W_1^\perp$ and consider the equation
\begin{equation}\label{invert}
H_*w =  g.
\end{equation}
Using  \eqref{eq:eigenvalues}, the equation \eqref{invert} is uniquely solved  by the function
$$w(s)=\sum_{\ell \in \N \setminus \{1\}}w_\ell v_\ell(s),$$ where
\begin{equation}\label{eq:coeff}
w_\ell=\frac{1}{\pi \lambda_\ell(T_*)}  \langle g, v_\ell  \rangle_{L^2_{p,e}}, \quad  \ell\ne 1.
\end{equation}
In addition,
\begin{align}\label{eqchatacte3}
\sum_{\ell \in \N \setminus \{1\}} (1+\ell^2)^{2} \langle w, v_\ell
\rangle_{L^2_{p,e}} ^2 =\frac{1}{\pi } \sum_{\ell \in \N \setminus \{1\}} \frac{(1+\ell^2)}{ \lambda_\ell^2(T_*)} (1+\ell^2)  \langle g, v_\ell
\rangle_{L^2_{p,e}} ^2.
\end{align}
Since  $g\in C^{1, \alpha}_{p,e}(\R) \subset H^{1}_{loc}(\R)$, we have  $g\in H^{1}_{p,e}$. This combined with  \eqref{eqchatacte}  and the first asymptotic in Lemma  \ref{dervtildeV}  allow to see that   the right hand side in \eqref{eqchatacte3} is bounded, which implies   $w\in H^{2}_{p,e}$.  We now show that $w\in  C^{2,\alpha}_{p,e}(\partial \O_1)$.

Recalling Proposition \ref{outer2} \eqref{eqlinear0},  \eqref{eq:coeff} reads
\begin{align}\label{eqlidotu2}
(N-3)\nabla \dot{u}(e_1, \tau)\cdot(-e_1,0)=(N-3)(N-2)w(\tau)+g(\tau) 
	\end{align}
where
$\dot{u} $ is the unique even and $2\pi$ periodic solution in $\tau$ of
\begin{equation}\label{eqlidotu3}
	\begin{cases}
	\Delta \dot{u} =0& \quad  \textrm { in } \quad   \O_{1}  \vspace{3mm} \\
	\dot{u}(y,\tau)=w( \tau ) & \quad  \textrm { in } \quad  \partial  \O_{1}  \vspace{2mm} \\
	\dot{u}\rightarrow0 &\emph{as $|y|\rightarrow\infty$ uniformly in $\tau\in\mathbb{R}.$}\\
	\end{cases}
	\end{equation}
Furthermore  since  $w \in H^{2}_{p,e}$, we have   by  standard elliptic regularity that $\dot{u} \in  W^{2,2}_{loc}( \O_{1}) $.
 	
We now show that
\begin{equation}\label{eqregular11}
 \dot{u}\in C^{2,\alpha}_{p,e}(\ov \O_1).
\end{equation}
The fact that  \eqref{eqregular11} holds  follows from a  similar argument   as in the proof of  \cite[Proposition 4.1]{Fall-MinlendI-Weth} (see also  the proof of  \cite[Proposition 5.1]{Fall-MinlendI-Weth2}). We give the details here for the reader's convenience. The  regularity property in  \eqref{eqregular11} is obtained from  \cite[Theorem 6.3.2.1]{Grisvard} once we  show  that
\begin{equation}\label{eqre2233}
 \dot{u}\in W^{2,p}_{loc}(\O_1) \ \mbox{ for any } p\in(1,\infty)
 \end{equation}
Indeed if  \eqref{eqre2233} holds then by Sobolev embedding, we get   $\dot{u}\in C^{1,\alpha}_{p,e}(\ov \O_1)$ and hence  $ w\in C^{1,\alpha}_{p,e}(\partial \O_1)$  by  \eqref{eqlidotu3}. Then  applying  \cite[Theorem 6.3.2.1]{Grisvard}  with the order $d=1$ to the boundary operator (see  \cite[Section 2.1]{Grisvard}) yields \eqref{eqregular11}.

To see \eqref{eqre2233}, we prove by induction that
\begin{equation}\label{eq.errt2}
\dot{u}\in W^{2,p_{\ell }}_{loc}(\O_1) 
 \end{equation}
for a sequence of numbers $p_\ell \in [2, \infty)$ with $p_0 = 2$ and $p_{\ell+1} \geq \frac{ N-1 }{N-p_\ell }p_\ell $ for $\ell \geq 0.$
Clearly, it holds for $p_0=2$. Let us now assume that \eqref{eq.errt2} holds for some $p_\ell \geq 2.$

We consider the following two cases:\\
\textbf{Case 1:} $p_\ell < N.$ By the trace Theorem, \cite[Theorem 5.4]{Adams},  we can get that
$$\dot{u}|_{\partial\O_1}\in W^{1,p_{\ell +1}}_{loc}(\partial\O_1)\ \mbox{ with } p_{\ell +1} :=\frac{ N-1 }{N-p_\ell }p_\ell  \geq \frac{N-1}{N-2}p_\ell ,$$
therefore $w\in W^{1,p_{\ell+1}}_{loc}(\partial\O_1).$ Since $g\in C^{1, \alpha}(\R)\subset W^{1,p_{\ell +1}}_{loc}(\partial\O_1)$ and by \eqref{eqlidotu2},
\begin{equation*}
(N-3) \partial_\nu \dot{u}( \pm e_1, \tau) + \dot{u} ( \pm e_1, \tau)  = f(\tau),
	\end{equation*}
where
$$f(\tau):= (1+ (N-3)(N-2))w(\tau )+g(\tau ) \in  W^{1,p_{\ell +1}}_{loc}(\partial\O_1).$$  Therefore,  $\dot{u}\in W^{2,p_{\ell+1}}_{loc}(\O_1)$  by  \cite[Theorem 2.4.2.6]{Grisvard}.\\
\textbf{Case 2:} $p_\ell \geq N.$ The trace theorem implies that $w\in W^{1,p}_{loc}(\partial\O_1)$ for any $p > 2,$ and then we repeat the above argument to  deduce  that $\dot{u}\in W^{2,p_{\ell+1}}_{loc}(\O_1)$ for arbitrarily chosen $p_{\ell+1} \geq \frac{ N-1 }{N-2}p_\ell.$

Finally, we conclude that \eqref{eqre2233}  holds and \eqref{eqregular11} follows. By passing to the trace, we see with  \eqref{eqregular11} that $w\in  C^{2,\alpha}_{p,e}(\partial \O_1)$ and the proof is complete.
\QED
\end{proof}
We are now in position to apply the Crandall-Rabinowitz theorem \cite[Theorem 1.7]{M.CR}, which will give rise to the following
bifurcation property.

\subsection*{Proof of Theorem \ref{Theo1} (completed)}
We define
\begin{align*}
\mathcal{X}^{\perp} := \left\lbrace  v \in X : \int^{\pi}_{-\pi} v(\tau) \cos(\tau)\,d\tau =0\right\rbrace.
\end{align*}
By Proposition  \ref{prop:All-HYPcRANDAL}  and  the Crandall-Rabinowitz Theorem (see \cite[Theorem 1.7]{M.CR}), we then find ${\e}>0$ and a smooth curve
$$
(-{\e },{\e }) \to  (0,+\infty) \times \mathcal{U}\subset  \R_+ \times X, \qquad s \mapsto (T(s),\varphi_s)
$$
such that
\begin{enumerate}
\item[(i)] $F(T(s),\varphi_s)=0$ for $s \in (-{\e },{\e })$,
\item[(ii)] $ T (0)= T_{*}$ and
\item[(iii)]  $\varphi _s = s \cos(\cdot)+ s v _s $ for $s \in (-\e ,\e )$ with a smooth curve
$$
(-{\e },{\e }) \to \mathcal{X} ^{\perp}, \qquad s \mapsto v _s
$$
satisfying $ v_0 =0$
and
$$\int^{\pi}_{-\pi} v_s (\tau) \cos(\tau)\,d\tau=0.$$
\end{enumerate}
Finally, since  $F(T(s),\varphi_s)=0$ for $s \in (-{\e },{\e })$,  we see from  \eqref{CR-map} and  Lemma~\ref{lem:Oper-inv} that  the  function $w_s:=w_{T_s, \varphi_s}$ solves \eqref{eq:overcondi22}. Furthermore  recalling  \eqref{eqans}, the function
$$
u_s(z,t)= w_s\Bigl(\zeta( |z|^2,\varphi_s(\frac{2\pi t}{T_s}))z, \frac{2\pi t}{T_s}\Bigl)
$$
solves  \eqref{eq:ovder3} on $\O_{T_s, \varphi_s }$. The proof is  complete. \QED
\appendix
\section{Scale invariant Hölder estimates for solutions of the Poisson equation}
\label{sec:Schauderes}
In this section, we recall some well-known Hölder estimates for solutions of the Poisson equation
$\Delta u=f$, and we reformulate them in a scale-invariant way.
We  first recall the following classical  regularity  results (see \cite[Theorem 4.6]{GilbardTru} and  \cite[Theorem 6.6]{GilbardTru}).
\begin{Lemma} \label{Interestit}
 Let  $f \in C^{0,\alpha}(B_{1})$ and   $u \in C^{2,\alpha}(B_1)$ solve  the equation $-\Delta u = f$ in $B_1$. Then there exists a constant $C=C(N,\alpha)>0$ such that
  \begin{align}
    \label{eqe}
   \| u\|_{C^{2,\alpha} (B_{1/2})} \le C ( \|f\|_{L^\infty(B_1)} +[f]_{C^{0,\alpha}(B_1)} +  \|u_R\|_{L^\infty(B_1)}).
  \end{align}
\end{Lemma}

\begin{Lemma} \label{Boundaryest}
Let  $f \in C^{0,\alpha}(\ov B_{1})$ and   $u \in C^{2,\alpha}(\ov B_1)$ solve  the equation $-\Delta u = f$ in $B_1$.

Assume there exists $\varphi\in  C^{2,\alpha}(\ov B_1)$ such that $u=\varphi$ on $\partial B_1$.
 Then there exists a constant $C=C(N,\alpha)>0$ such that
  \begin{align}
    \label{eqe2}
   \| u\|_{C^{2,\alpha} (B_{1})} \le C (  \| \varphi\|_{C^{2,\alpha} (B_{1})}  +\|f\|_{L^\infty(B_1)} +[f]_{C^{0,\alpha}(B_1)} +  \|u_R\|_{L^\infty(B_1)}).
  \end{align}
\end{Lemma}

\begin{Lemma}
  \label{basic-poisson}
  Let $z \in \R^N$, $R>0$ and $f \in C^{0,\alpha}(B_{R}(z))$.\\
  Moreover, let $u \in C^{2,\alpha}_{loc}(B_R(z)) \cap L^\infty(B_R(z))$ solve $-\Delta u = f$ in $B_R(z)$. Then there exists a constant $C>0$, independent of $R$, with the property that
  \begin{align}
    \label{eq:R-dependent}
  &\sum_{i = 0}^2 R^i \|\nabla^i u\|_{L^\infty(B_{\frac{R}{2}}(0))} + R^{2+\alpha}[\nabla^2 u]_{C^{0,\alpha} (B_{\frac{R}{2}}(0))}\nonumber\\
 &\le C( R^2 \|f\|_{L^\infty(B_{R}(0))} + R^{2+\alpha} [f]_{C^{0,\alpha}(B_{R}(0))} +  \|u\|_{L^\infty(B_{R}(0))}).
  \end{align}
Furthermore if

$f \in C^{0,\alpha}(\ov B_{R}(z))$  and  $u \in C^{2,\alpha}_{loc}(\ov B_R(z)) \cap L^\infty(B_R(z))$ solve $-\Delta u = f$ in $B_R(z)$, with $u=\varphi$ on $\partial B_R(z)$,  then
  \begin{align}
    \label{eq:R-dependen2}
&\sum_{i = 0}^2 R^i \|\nabla^i u\|_{L^\infty(B_{R}(0))} +   R^{2+\alpha}[\nabla^2 u]_{C^{0,\alpha} (B_{R}(0))}\nonumber\\
&\le C( \sum_{i = 0}^2 R^i \|\nabla^i \varphi \|_{L^\infty(B_{R}(0))} +   R^{2+\alpha}[\nabla^2 \varphi]_{C^{0,\alpha} (B_{R}(0))} +  \|u\|_{L^\infty(B_{R}(0))} )\nonumber\\
 &+  C(R^2 \|f\|_{L^\infty(B_{R}(0))} + R^{2+\alpha} [f]_{C^{0,\alpha}(B_{R}(0))} ).
  \end{align}
  Here, $[\cdot]$ is the Hölder semi-norm defined in Definition~\ref{main-definition-hoelder-etc}.
\end{Lemma}

\begin{proof}
  Without loss of generality, we may take $z=0$.  Hence we assume that
  \begin{equation}
    \label{eq:poisson-assumption}
  -\Delta u = f\qquad \text{in $B_R(0)$,}
  \end{equation}
and we let $u_R, f_R : B_1(0) \to \R$ be defined by $u_R(x)=  u(Rx)$, $f_R(x)=  f(Rx)$. Then we have
  \begin{align*}
    &\nabla u_R = R (\nabla u)(R \,\cdot\,)\qquad \text{in $B_1(0)$,}\\
    &\nabla^2 u_R = R^2 (\nabla^2 u)(R \,\cdot\,)\qquad \text{in $B_1(0)$.}
  \end{align*}
  Since by (\ref{eq:poisson-assumption}) we have
  \begin{equation}
    \label{eq:poisson-assumption1}
  -\Delta u_R = R^2 f_R\qquad \text{in $B_1(0)$,}
  \end{equation}
we can apply   \eqref{eqe} to get
\begin{align*}
   \| u_R\|_{C^{2,\alpha} (B_{1/2})} \le C ( \|R^2 f_R\|_{L^\infty(B_1)} +[R^2 f_R]_{C^{0,\alpha}(B_1)} +  \|u_R\|_{L^\infty(B_1)}),
  \end{align*}
where $C$ is a constant independent of $R$ and $f$.
Combining this estimate with the scaling identities listed above, we obtain
  \begin{align*}
  &\sum_{i = 0}^2 R^i \|\nabla^i u\|_{L^\infty(B_{\frac{R}{2}}(0))} + R^{2+\alpha}[\nabla^2 u]_{C^{0,\alpha} (B_{\frac{R}{2}}(0))}\nonumber\\
 &\le C( R^2 \|f\|_{L^\infty(B_{R}(0))} + R^{2+\alpha} [f]_{C^{0,\alpha}(B_{R}(0))} +  \|u\|_{L^\infty(B_{R}(0))}),
  \end{align*}
which gives \eqref{eq:R-dependent}  in the case $z=0$. Similarly, we obtain  \eqref{eq:R-dependen2} using Lemma \ref{Boundaryest}.
\end{proof}\QED

\section{Identities and inequalities  involving modified Bessel functions}\label{appendix}
Here  we collect  some  properties on the  modified  Bessel functions $K_{\nu}$.
\subsection{General properties}\label{appendix1}
For $\nu \ge 0$, the  modified Bessel function $K_\nu$ is defined on $(0,\infty)$ by  the integral representation
\begin{equation}\label{Bes2}
 K_{\nu}(x)= \int_{0}^{\infty} e^{-x \cosh (t)} \cosh (\nu t)  dt \qquad \textrm{ for  $x > 0$}.
\end{equation}

\subsection{Derivatives}\label{appendix2}
For all  $x \in(0, + \infty)$, we have
\begin{align}
x \frac{K'_{\nu+1}(x)}{K_{\nu}(x)}  &=-x -(\nu+1) \frac{K_{\nu+1}(x)}{K_{\nu}(x)} , \label{relateK1}\\
x \frac{K'_{\nu}(x)}{K_{\nu}(x)}  &=\nu  -x\frac{K_{\nu+1}(x)}{K_{\nu}(x)} \label{relateK2},
\end{align}
see e.g.   \cite[Page 6]{HangYu}) or  and \cite{BariczPonnusamy} and \cite{El}.
\subsection{Asymptotic behaviour}\label{appendix2-asym}

Asymptotics of   $K_{\nu}$ are given e.g. in \cite[Page 4]{HangYu}). In particular, we have
for all $\nu > 0,$
 \begin{align}
K_\nu(x)\sim \frac{\sqrt{\pi}}{\sqrt{2}}x^{-\frac{1}{2}} e^{-x} \quad \textrm{as} \quad  x\longrightarrow +\infty  \label{eqasymp1} \\
K_\nu(x)\sim \frac{1}{2}\Gamma(\nu)\Bigl(\frac{x}{2}\Bigl)^{-\nu}  \quad \textrm{as} \quad  x\longrightarrow 0.  \label{eqasymp2}
\end{align}

\subsection{Inequalities}\label{appendix3}
The following inequality identity can be found in  \cite{HangYu}: For every $\rho>0$  and $\nu\geq 0$,

\begin{align} \label{eq:compI1I00}
& \frac{K_{\nu+1}(\rho)}{ K_{\nu}(\rho)}> \frac{\nu+\sqrt{\rho^2+\nu^2}}{\rho }.
\end{align}

\end{document}